\documentclass[a4paper,reqno]{amsart}
\usepackage[latin1]{inputenc}
\usepackage{amsfonts}
\usepackage{amsmath,latexsym,amssymb,amsfonts, amsthm,amsbsy}
\usepackage{esint}
\usepackage{cite}
\usepackage{graphicx}
\usepackage{amscd}
\usepackage{color}
\usepackage{bm}             
\usepackage{enumerate}
\usepackage[dvips]{epsfig}
\usepackage{psfrag}
\usepackage{epsfig}
\usepackage{mathrsfs}  
\usepackage{bm}
\usepackage{cases}

\usepackage{empheq}

\usepackage{subcaption}

\usepackage{tikz}

\usepackage[
  hmarginratio={1:1},     
  vmarginratio={1:1},     
  textwidth=16cm,        
  textheight=21cm,
  heightrounded,          
]{geometry}

\usepackage{graphicx,color}
\usepackage[colorlinks]{hyperref}
\hypersetup{linkcolor=blue,citecolor=blue,filecolor=black,urlcolor=blue}

\newtheorem{theorem}{Theorem}
\newtheorem{corollary}[theorem]{Corollary}
\newtheorem{proposition}[theorem]{Proposition}
\newtheorem{lemma}[theorem]{Lemma}
\theoremstyle{definition}
\newtheorem{remark}{Remark}

\numberwithin{theorem}{section}

\numberwithin{remark}{section}

\numberwithin{equation}{section}

\newcommand{\R}{\mathbb{R}}
\newcommand{\N}{\mathbb{N}}

\newcommand{\beq}{\begin{equation}}
\newcommand{\eeq}{\end{equation}}

\newcommand{\beqn}{\begin{equation*}}
\newcommand{\eeqn}{\end{equation*}}

\newcommand{\Mb}{{\mathbb{M}}}

\title[The Lane-Emden system on Cartan-Hadamard manifolds]{The Lane-Emden system on Cartan-Hadamard manifolds: asymptotics and rigidity of radial solutions}

\author[M. Muratori]{Matteo Muratori}\thanks{}
\address{Matteo Muratori \newline \indent
	Dipartimento di Matematica,  Politecnico di Milano  \newline \indent
	Piazza Leonardo da Vinci 32, 20133 Milano, Italy}
\email{matteo.muratori@polimi.it}

\author[N. Soave]{Nicola Soave}\thanks{}
\address{Nicola Soave \newline \indent
	Dipartimento di Matematica ``Giuseppe Peano'',  Universit\`a degli Studi di Torino  \newline \indent
	Via Carlo Alberto 10, 10123 Torino, Italy}
\email{nicola.soave@unito.it}

\makeatletter
\@namedef{subjclassname@2020}{%
  \textup{2020} Mathematics Subject Classification}
\makeatother

\keywords{Cartan-Hadamard manifolds; Lane-Emden system; Radial solutions; Stochastic completeness.}
\subjclass[2020]{Primary: 58J05; 35B53; 35J47; 35J91. Secondary: 58J70; 34A12; 34A34; 34C40; 34D05.}

\begin{document}

\begin{abstract}
We investigate existence and qualitative properties of globally defined and positive radial solutions of the Lane-Emden system, posed on a Cartan-Hadamard model manifold $ \Mb^n $. We prove that, for critical or supercritical exponents, there exists at least a one-parameter family of such solutions. Depending on the stochastic completeness or incompleteness of $ \Mb^n $, we show that the existence region stays one dimensional in the former case, whereas it becomes two dimensional in the latter. Then, we study the asymptotics at infinity of solutions, which again exhibit a dichotomous behavior between the stochastically complete (where both components are forced to vanish) and incomplete cases. Finally, we prove a rigidity result for finite-energy solutions, showing that they exist if and only if $ \Mb^n $ is isometric to $ \R^n $.
\end{abstract}

\maketitle


\section{Introduction}\label{sec: intro}

The Lane-Emden equation 
\beq\label{LE eq}
-\Delta u= u^p \, , \quad u>0 \, , 
\eeq
with $p>0$, is the prototype of semilinear elliptic equations, and has played a central role in the development of several tools in the analysis of nonlinear PDEs. Its system counterpart, known in the literature as \emph{Lane-Emden system}, that is
\beq\label{LE sys}
\begin{cases}
-\Delta u = v^q \,   \\
-\Delta v = u^p \,  \\
u,v >0 \, , 
\end{cases}
\eeq
with $p,q>0 $, has also received a lot of attention in the recent years, but is far less understood. The purpose of this paper is to study existence and qualitative properties of \emph{radial solutions} to \eqref{LE sys} posed on a \emph{Cartan-Hadamard manifold} $ \Mb^n $, in the \emph{critical or supercritical regime} of the exponents (see below). Recall that a Cartan-Hadamard manifold is a complete and simply connected Riemannian manifold with nonpositive sectional curvature. An important feature of this type of manifolds consists in the possibility of writing \emph{global polar coordinates} centered at \emph{any} reference point $ o \in \Mb^n $, as the well-known Cartan-Hadamard theorem entails that the exponential map at $o$ is  a diffeomorphism between $ \R^n $ and $ \Mb^n $.

Thus, in order to understand existence and properties of solutions to \eqref{LE sys} on Cartan-Hadamard manifolds, and having in mind the Euclidean case, 
it appears natural to focus on the radial problem at first, upon requiring in addition that the ambient space itself is spherically symmetric. In this framework, we establish the existence of positive radial solutions on any Cartan-Hadamard \emph{model} manifold (we refer to Subsection \ref{mb} for definitions and details on the geometric background). As a second step, focusing on uniqueness or multiplicity, and on the limit behavior of such solutions, we discover an interesting dichotomy: if the underlying manifold $\Mb^n$ is \emph{stochastically complete}, then the scenario is Euclidean like, namely we prove a uniqueness result and show that all solutions are such that both $u$ and $v$ vanish at infinity (see Theorem \ref{monotone-coro}); if instead $\Mb^n$ is \emph{stochastically incomplete}, then we show a new phenomenon of multiple existence of positive solutions with strictly positive limits at infinity (see Theorem \ref{teo-cs-interv}). Finally, we establish a strong rigidity result in terms of the natural energies involved when dealing with \eqref{LE sys}: either they are all finite, and in this case the underlying manifold is necessarily $\R^n$, the problem is critical, and $(u,v)$ belongs to a $1$-parameter family, or they are all infinite (see Theorem \ref{full-rigidity-proof}). As a consequence of our methods of proof, we can actually extend all of our main results to a suitable class of Riemannian models that are not necessarily Cartan-Hadamard (see Corollary \ref{CH-remove}). 

%

\subsection{Motivation and the state of the art}\label{msa}

Existence and qualitative properties of solutions to \eqref{LE eq}, posed in the Euclidean space $\R^n$, $n \ge 2$, are by now well understood: in the \emph{subcritical} regime, namely $1<p+1<2^*:=\frac{2n}{n-2} $ (with $2^\ast=\infty$ if $n=2$), the problem has no classical solutions, regardless of radial symmetry. In the \emph{critical} case $p+1=2^*$ such solutions exist, are radially symmetric, and correspond to the extremals of the Sobolev inequality; in particular, they belong to the Sobolev space $D^{1,2}(\R^n)$, defined as the completion of $C^\infty_c(\R^n)$ with respect to the norm 
\[
\|u\|_{D^{1,2}(\R^n)}^2 := \int_{\R^n} |\nabla u|^2 \, dx \, .
\]
In the \emph{supercritical} regime $p+1>2^*$, radial classical solutions still exist, decay to $0$ at infinity, but do not belong to the \emph{energy space} $D^{1,2}(\R^n)$. We refer the interested reader to the excellent monograph \cite{QS}, and to the corresponding bibliography, for further details.

The situation can be considerably different if \eqref{LE eq} is posed on a  Cartan-Hadamard manifold $\Mb^n$. A first remarkable difference is that positive radial solutions may exist even in the subcritical regime, both with finite and infinite $D^{1,2}(\Mb^n)$ norm; for instance, this is the case if $\Mb^n \equiv \mathbb{H}^n$ is the $n$-dimensional hyperbolic space \cite{BGGV, MaSa}, or if $\Mb^n$ is a more general \emph{model manifold} satisfying suitable assumptions \cite[Theorems 2.5 and 2.7]{BeFeGr}. Concerning the critical or supercritical regimes, the situation is somehow more rigid, in the following sense. If $u$ is a radial solution to \eqref{LE eq} on a Cartan-Hadamard model manifold $\Mb^n$, with $p + 1 \ge 2^\ast  $ and $\|u\|_{D^{1,2}(\R^n)}<+\infty$, then $\Mb^n$ is necessarily isometric to $\R^n$ and $u$ is therefore an Aubin-Talenti function \cite[Theorem 1.3]{MS} (see also \cite[Theorem 2.4]{BeFeGr} for a related result obtained under additional assumptions on $\Mb^n$, and \cite{MK} for a previous rigidity result concerning solutions that minimize the Sobolev quotient). Furthermore, the asymptotic behavior of radial solutions is strongly affected by the global geometric properties of the underlying manifold: if $\mathbb{M}^n$ is \emph{stochastically complete}, then all radial solutions tend to $0$ at infinity; otherwise, if it is \emph{stochastically incomplete}, each solution converges to a strictly positive constant at infinity \cite[Theorem 1.4]{MS}. Additional asymptotic estimates can be found in \cite[Theorem 1.5]{MS} and \cite[Theorem 2.4]{BeFeGr}.

We finally refer to \cite{BaKr, CatMon, FogMalMaz, GidSpr} and references therein for results regarding \eqref{LE eq} and related inequalities posed on manifolds with \emph{nonnegative Ricci curvature}, namely the case complementary to ours. 

\medskip 

Concerning system \eqref{LE sys}, the problem in $\R^n$ presents several analogies with the corresponding scalar case. In particular, one can naturally introduce a subcritical regime 
\beqn\label{subcrit}
p, q >0 \, , \qquad \frac{1}{p+1}+\frac{1}{q+1} > \frac{n-2}{n} \, ,
\eeqn
a critical regime
\beq\label{crit}
p, q >0 \, , \qquad \frac{1}{p+1}+\frac{1}{q+1} = \frac{n-2}{n} \, ,
\eeq
and a supercritical regime
\beq\label{supercritical}
p, q >0 \, , \qquad \frac{1}{p+1}+\frac{1}{q+1} < \frac{n-2}{n} \, .
\eeq 
For subcritical exponents, radial positive solutions do not exist \cite{Mit}, and it is conjectured that positive classical solutions do not exist at all. This has been rigorously proved only up to dimension $n=4$ \cite{Sou} (see also \cite{BuMa, PoQuSo, ReZo, SeZo2}). On the other hand, for critical or supercritical exponents radial positive solutions do exist \cite{Lio, SerZou0, SerZou}. In the critical case, they correspond to extremals for higher-order Sobolev inequalities \cite{Lio} (see also \cite{Wan}), while in the supercritical regime it is not expected that they belong to any natural Sobolev space.

As far as we know, problem \eqref{LE sys} on Cartan-Hadamard (model) manifolds is untouched. The purpose of this paper is thus to investigate existence and qualitative properties of \emph{radial positive solutions}, in the \emph{critical} or \emph{supercritical} regimes. This implies working in spatial dimension $ n \ge 3 $, otherwise conditions \eqref{crit} and \eqref{supercritical} are empty. As concerns the subcritical regime, that we do not address here, it is natural to wonder whether or not, on suitable Cartan-Hadamard manifolds that significantly differ from $ \R^n $, globally positive solutions can exist. If we think of what happens in the Euclidean case (non-existence, as recalled above), and the results in \cite{MaSa} and \cite{BeFeGr} (existence for \eqref{LE eq} on the hyperbolic space and on many other model manifolds), the answer still depends on the specific analytic-geometric properties of $\Mb^n$. 

\medskip

In order to state our main results in a precise form, it is necessary to recall first some standard definitions and introduce notations accordingly.

\subsection{Main results and basic notions} \label{mb}
We say that a noncompact Riemannian manifold $ \Mb^n $ is a \emph{model} if there exists a pole $o \in \Mb^n$ such that its metric $g$ is given, in polar (or spherical) global coordinates about $o$, by 
\beqn\label{metric}
g \equiv dr \otimes dr + \psi^2(r) \, g_{\mathbb{S}^{n-1}} \, , 
\eeqn
where $r$ is the Riemannian distance of a point $ x \equiv  (r,\theta) \in \mathbb{R}^+ \times \mathbb{S}^{n-1} $ from $o$, $g_{\mathbb{S}^{n-1}}$ stands for the usual round metric on the unit sphere $ \mathbb{S}^{n-1} $ and $\psi $ is a $ C^1([0,+\infty)) \cap C^\infty((0,+\infty)) $ positive function with $\psi(0) = 0$ and $\psi'(0) = 1$ (for a more complete introduction to model manifolds we refer the reader \emph{e.g.}~to \cite{GW,Gri}). The Cartan-Hadamard assumption in this case is equivalent to the fact that $ \psi $ is in addition \emph{convex}, due to the explicit expression of the sectional curvatures in terms of $ \psi $, see for instance \cite[Section 2]{GMV17}. A prototypical example is represented by the choice $\psi(r) = \sinh r$, which gives rise to a well-known realization of the \emph{hyperbolic space} $\mathbb{H}^n$, whose sectional curvature is identically equal to $ -1 $.  

For future convenience, we define 
\begin{equation}\label{def-theta}
\Theta(r) := \frac{1}{\psi^{n-1}(r)} \int_0^r \psi^{n-1} \, ds\qquad \forall r > 0 \, ;
\end{equation}
namely, $\Theta$ is the function accounting for the volume-surface ratio of geodesic balls centered at the pole. Its importance, for our purposes, is due to the fact that a model manifold $\Mb^n$ (not necessarily of Cartan-Hadamard type) is \emph{stochastically complete} if $\Theta \not \in L^1(\R^+)$, whereas it is \emph{stochastically incomplete} if $\Theta \in L^1(\R^+)$. This dichotomy will have a key role in our results. We refer to \cite[Sections 3 and 6]{Gri} for a deeper discussion, and we also point out the recent papers \cite{GIM, GIMP} for new nonlinear analytic characterizations of stochastic (in)completeness for general manifolds.

\medskip

By writing \eqref{LE sys} in polar coordinates, it is not difficult to check (see for instance \cite{BeFeGr} or \cite{MS} for the details) that a radial solution is (represented by) a regular enough function $(u,v): [0,+\infty) \to \R^2$ solving the Cauchy problem
\beq\label{Cauchy pb}
\begin{cases}
\left(\psi^{n-1}\,u'\right)' + \psi^{n-1} \, |v|^{q-1}  v =0 & \text{ for $r>0$}  \\
\left(\psi^{n-1}\,v'\right)' + \psi^{n-1} \, |u|^{p-1} u =0 & \text{ for $r>0$} \\
u'(0) = 0 = v'(0)  \\
u(0) = \xi \, , \quad v(0) = \eta \, ,
\end{cases}
\eeq
for some initial data $(\xi, \eta) \in (0,+\infty)^2$. Note that, although we are only interested in positive solutions, for technical reasons it is necessary to be able to deal with sign-changing solutions as well, whence the replacement of $ v^q $ and $ u^p $ with $ |v|^{q-1}  v $ and $ |u|^{p-1} u $, respectively. The fact that \eqref{Cauchy pb} gives rise to an everywhere positive solution is a highly nontrivial issue, which is actually false in general, and will be thoroughly addressed in Section \ref{esistenza-base}. In what follows, we will say that $(u,v)$ is a (radial) \emph{globally positive solution} if it solves \eqref{Cauchy pb} for every $r>0$ and $u, v>0$ on $[0,+\infty)$. Clearly, any such a solution solves the Lane-Emden system \eqref{LE sys}. 

At this point it is worth recalling that, in the Euclidean setting, from well-known results due to Serrin and Zou \cite{SerZou0, SerZou} (see also \cite{Lio}) the system \eqref{Cauchy pb} for $ \Mb^n \equiv \R^n $ and in the critical-supercritical regime
\begin{equation}\label{crit-sup}
\frac{1}{p+1} + \frac{1}{q+1} \le \frac{n-2}{n}
\end{equation}	
admits a globally positive solution if and only if the initial data $ (\xi,\eta) $ satisfy the explicit relation
\beq\label{scaling-curve}
\eta \equiv \eta(\xi) = c \, \xi^\frac{p+1}{q+1} \, ,
\eeq
where $ c $ is a positive constant depending only on $  p,q,n $. In particular, we observe that $ \xi \mapsto \eta(\xi)$ is a strictly increasing and continuous bijection of $(0,+\infty)$ into itself. We stress that the specific form \eqref{scaling-curve} of the function $ \eta(\xi) $ is crucially related to the natural \emph{scaling properties} of \eqref{LE sys} in $ \R^n $, which however fail on model manifolds. Indeed, as we will see in a moment, such function is not explicit and its very definition, \emph{i.e.}~the fact that for every $ \xi>0 $ there exists a \emph{unique} value of $ \eta $ ensuring global positivity, strongly depends on the stochastic completeness or incompleteness of $\Mb^n$. 

\medskip

Before stating our main results, for notational convenience, for any globally positive solution of \eqref{Cauchy pb} we set
$$
\ell_u := \lim_{r \to +\infty} u(r)  \, , \qquad \ell_v := \lim_{r \to +\infty} v(r) \, ,
$$
the existence of such limits being guaranteed by the monotonicity of both components, which readily follows from the differential equations in \eqref{Cauchy pb} (see Section \ref{sec: prel}).

\begin{theorem}[Globally positive solutions for stochastically complete manifolds]\label{monotone-coro}
Let $ \Mb^n $ ($ n \ge 3 $) be a Cartan-Hadamard model manifold associated to a function $ \psi $ with $ \Theta \not \in L^1(\R^+) $. Let $ p,q>0 $ fulfill \eqref{crit-sup}. Then, for each $ \xi>0 $ there exists one and only one  value $ \eta \equiv \eta(\xi)>0 $ such that $ (\xi,\eta) $ gives rise to a globally positive solution $ (u,v) $ to \eqref{Cauchy pb}, which satisfies
$$
\ell_u=\ell_v=0 \, .
$$
Moreover, the function $ \xi \mapsto \eta(\xi) $ is a continuous and strictly increasing bijection of $(0,+\infty)$ into itself. 
\end{theorem}

Hence, in the stochastically complete case the situation is Euclidean like, since there exists a specific continuous curve of initial data that give rise to globally positive solutions, except that it has no more the explicit expression \eqref{scaling-curve}. Furthermore, both components of the solution vanish as $ r \to +\infty  $. This agrees with the results of Serrin and Zou in \cite{SerZou0, SerZou}. Instead, in the stochastically incomplete case the scenario is more complicated and marks a striking difference with respect to the Euclidean framework. 

\begin{theorem}[Globally positive solutions for stochastically incomplete manifolds]\label{teo-cs-interv}
Let $ \Mb^n $ ($ n \ge 3 $) be a Cartan-Hadamard model manifold associated to a function $ \psi $ with $ \Theta \in L^1(\R^+) $. Let $ p,q>0 $ fulfill \eqref{crit-sup}. Then there exist two functions $ \eta_{m} , \eta_{M} $ which are continuous and strictly increasing bijections of $ (0,+\infty) $ into itself, satisfying
\beq\label{prop-etamm-bis}
 \eta_{m}(\xi) < \eta_{M}(\xi)  \qquad \forall \xi>0 \, , 
\eeq
\beq\label{prop-etamm-ter}
\limsup_{\xi \to +\infty} \left[ \eta_{M}(\xi) - \eta_{m}(\xi)  \right] < + \infty \,  , 
\eeq
such that for each $ \xi>0 $ problem \eqref{Cauchy pb} admits a globally positive solution $ (u,v) $ if and only if 
\beq\label{prop-etamm-2}
\eta_{m}(\xi) \le \eta \le \eta_{M}(\xi) \, .
\eeq
In addition, the following behavior at infinity holds: 
\begin{subequations}
\begin{empheq}[left={\empheqlbrace}]{alignat = 2}
& \ell_u>0 \, , \ \ell_v = 0  \qquad  \text{if } \eta=\eta_m(\xi) \, , \label{upvz}  \\ 
& \ell_u>0 \, , \ \ell_v > 0  \qquad  \text{if } \eta_m(\xi)<\eta<\eta_M(\xi) \, ,  \label{upvp} \\ 
& \ell_u=0 \, , \ \ell_v > 0  \qquad  \text{if } \eta=\eta_M(\xi) \, .  \label{uzvp}
\end{empheq}
\end{subequations}
\end{theorem}

Two symbolic instances of the global positivity region in the space of the parameters $ (\xi,\eta) $, associated with a stochastically complete and a stochastically incomplete Cartan-Hadamard model manifold, respectively, are depicted in Figure \ref{fig:figures} below.

\medskip


We now focus on the possible existence of radial \emph{finite-energy solutions} to \eqref{LE sys} in the critical or supercritical cases. We recall that for the scalar problem \eqref{LE eq} such solutions cannot exist unless $ \Mb^n \equiv \R^n $ (\emph{i.e.}~the manifold is isometric to the Euclidean space), as shown in \cite{MS}. Here we are able to reproduce the natural counterpart of this rigidity result for the Lane-Emden system. 

\begin{figure} 
\centering 
\quad \,
\begin{subfigure}[b]{0.4\textwidth} 
\centering
     \includegraphics[scale=0.9]{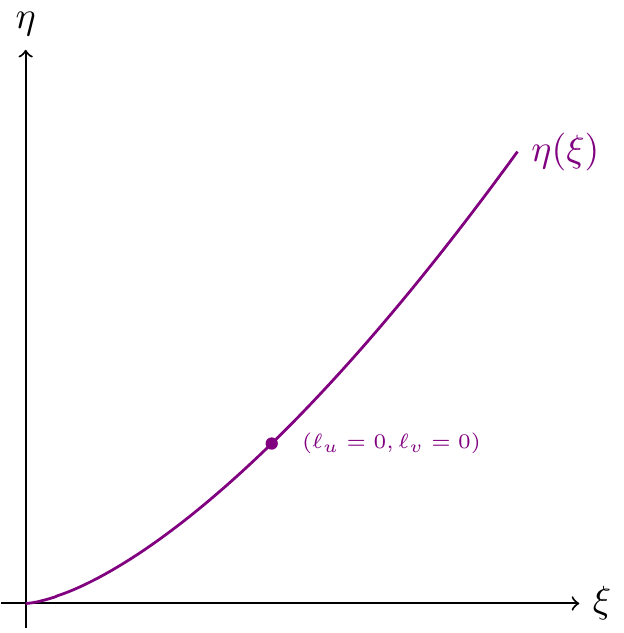} 
    \caption{$ \Theta \not \in L^1(\R^+) $}\label{fig:first}
\end{subfigure}
 \hfill   
\begin{subfigure}[b]{0.4\textwidth}    
\centering
    \includegraphics[scale=0.9]{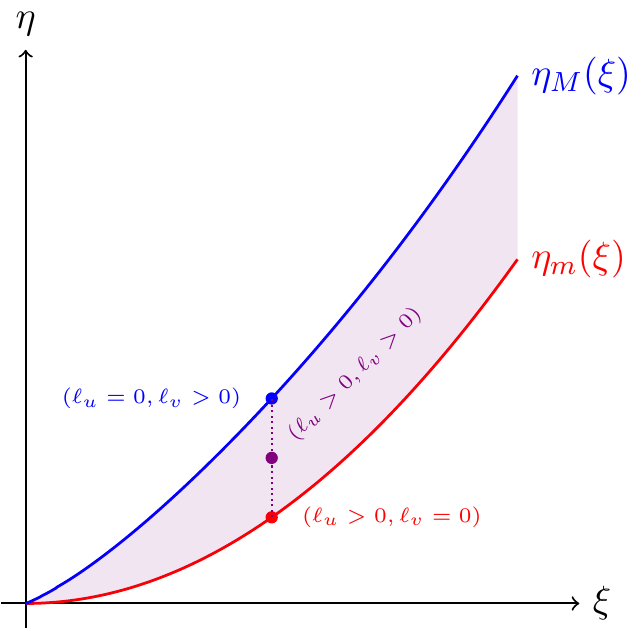} 
    \caption{$ \Theta \in L^1(\R^+) $}\label{fig:second} 
\end{subfigure}
\quad \,
\caption{\small The regions of existence of a globally positive solution in terms of the initial data $ (\xi,\eta) $, in the case of a stochastically complete (\textsc{\subref{fig:first}}) and incomplete (\textsc{\subref{fig:second}}) manifold, with the corresponding behavior of the limits at infinity $ (\ell_u,\ell_v) $ of the components.}
\label{fig:figures}
\end{figure}

\begin{theorem}[Energy rigidity]\label{full-rigidity-proof}
Let $\Mb^n$ ($ n \ge 3 $) be a Cartan-Hadamard model manifold associated to a function $ \psi $, and let $ p,q>0 $ fulfill \eqref{crit-sup}. Suppose that, for some $ (\xi,\eta) \in (0,+\infty)^2 $, there exists a radial solution $(u,v)$ to \eqref{LE sys} such that 
$$
\int_0^{+\infty} u' v' \, \psi^{n-1} \, ds < + \infty \quad \text{or} \quad \int_0^{+\infty} u^{p+1} \, \psi^{n-1} \, ds < + \infty  \quad \text{or} \quad \int_0^{+\infty} v^{q+1} \, \psi^{n-1} \, ds < + \infty \, .
$$
Then $ \Mb^n \equiv \R^n $, 
$$
\frac{1}{p+1} + \frac{1}{q+1} = \frac{n-2}{n} 
$$
and 
\beq\label{var-arg}
\int_0^{+\infty} u' v' \, \psi^{n-1} \, ds + \int_0^{+\infty} u^{p+1} \, \psi^{n-1} \, ds + \int_0^{+\infty} v^{q+1} \, \psi^{n-1} \, ds < + \infty \, .
\eeq 
\end{theorem}

\begin{remark}\label{R1}
We point out that on $\Mb^n \equiv \R^n$ radial solutions do comply with \eqref{var-arg} for every critical pair $(p,q)$, see \cite{Lio}. They can be obtained through a variational argument combined with a scaling property, and are characterized as extremals of higher-order Sobolev inequalities. As already observed, up to translations, radial solutions form a $1$-parameter family (the parameter being the value of one component, say $u$, at $r=0$), and their asymptotic behavior is understood \cite{HulVdV}. Differently from what happens in the scalar case \eqref{LE eq}, their explicit expression is however unknown in general. 

It is also worth to mention that any solution to \eqref{LE sys} on $\R^n$ with $(u,v) \in L^{p+1}(\R^n) \times L^{q+1}(\R^n)$ is radially symmetric \cite{CheLiOu}. This is proved by using an integral version of the moving planes method, and seems hardly adaptable on general manifolds.
\end{remark}


\medskip 

We now state a key preliminary proposition regarding the ``first zeros'' of a solution to \eqref{Cauchy pb}, which may be of independent interest and will play an important role in the proof of our main results.

\begin{proposition}\label{thm: non-ex bdd}
Let $\Mb^n$ ($ n \ge 3 $) be a Cartan-Hadamard model manifold associated to a function $ \psi  $, and let $ p,q>0 $ fulfill \eqref{crit-sup}. Then, given any $(\xi,\eta) \in (0,+\infty)^2$, problem \eqref{Cauchy pb} coupled with
\[
u( R )=v( R ) = 0 \qquad \text{for some $R>0$}
\]
has no positive solution on $(0,R)$.
\end{proposition}

Note that this can be read as a nonexistence result for the radial homogeneous Dirichlet problem in Riemannian balls. 

\medskip

We conclude with a result in the spirit of \cite[Theorem 2.2]{BeFeGr}, which is essentially a consequence of our methods of proof, showing that the Cartan-Hadamard assumption can be slightly relaxed.

\begin{corollary}\label{CH-remove}
Let $\Mb^n$ ($ n \ge 3 $) be a noncompact model manifold associated to a function $ \psi $, and let $ p,q>0 $ fulfill \eqref{crit-sup}. Then Theorems \ref{monotone-coro}, \ref{teo-cs-interv}, \ref{full-rigidity-proof} and Proposition \ref{thm: non-ex bdd} still hold provided the function
\beq\label{vol-convex}
\mathcal{V}(r) := \left( \int_0^r \psi^{n-1} \, ds \right)^{\frac{pq-1}{2(p+1)(q+1)}}  \qquad \forall r \in [0,+\infty)
\eeq
is convex.
\end{corollary}

\begin{remark}
In the critical case, the exponent in formula \eqref{vol-convex} attains its minimum value (within the critical-supercritical regime), which is precisely $ \frac{1}{n-1} $. After some routine calculations (see \emph{e.g.}~the proof of Proposition \ref{prop: der P} below), one can check that the convexity of $ \mathcal{V} $ is equivalent to 
$$
 \int_0^r \frac{\psi^n \, \psi''}{\left(\psi'\right)^2}\,ds \ge 0 \qquad \forall r \in (0,+\infty) \, .
$$
In particular, we emphasize the fact that there is room for $ \psi'' $ to be negative somewhere, which means that the (radial) curvatures of $ \Mb^n $ are allowed to be positive is some small region. Clearly, this is \emph{a fortiori} admissible in the supercritical regime.  
\end{remark}

\subsection{Paper organization} 
We devote Section \ref{sec: prel} to the proof of some useful inequalities and local existence results for the solutions to \eqref{Cauchy pb} (including Proposition \ref{thm: non-ex bdd}). In Section \ref{esistenza-base} we show that, for a given $ \xi>0 $, there exists at least one $ \eta>0 $ for which \eqref{Cauchy pb} yields a globally positive solution. This will require a number of preliminary technical tools. In Section \ref{gps} we prove our main results regarding the complete structure of the region of global positivity in the initial-data space $ (\xi,\eta) $ and the asymptotics of solutions as $ r \to +\infty $, that is Theorems \ref{monotone-coro} and \ref{teo-cs-interv}. Finally, Section \ref{rigid} contains the proof of the rigidity Theorem \ref{full-rigidity-proof}, and in Section \ref{CG} we establish the generalization of our main results stated in Corollary \ref{CH-remove}.

\medskip

For notational convenience, from here on we set $ a \vee b := \max\{ a,b \}  $ and $ a \wedge b := \min\{ a,b \}  $. 

\section{Preliminary properties of radial solutions}\label{sec: prel}

In this section, we first establish some basic local existence results for problem \eqref{Cauchy pb}, and then we focus on a key Pohozaev-type inequality and its consequences. 

From here on, unless otherwise specified, we take for granted that $ \Mb^n $ is a Cartan-Hadamard model manifold of dimension $ n \ge 3 $. In particular, we observe that the differential equations appearing in the system \eqref{Cauchy pb} can be rewritten as 
\beq\label{radial eq}
u'' + (n-1) \frac{\psi'}{\psi} \, u' + |v|^{q-1} v = 0 \, , \qquad v'' + (n-1) \frac{\psi'}{\psi} \, v' + |u|^{p-1} u = 0 \, ,
\eeq
or, by integrating,
\beq\label{eq int form}
\psi^{n-1}(r) \, u'(r) = - \int_0^r  |v|^{q-1}v \, \psi^{n-1}  \, ds \, , \qquad  \psi^{n-1}(r) \, v'(r) = - \int_0^r |u|^{p-1}u\, \psi^{n-1} \, ds \, .
\eeq
In the sequel, both \eqref{radial eq} and \eqref{eq int form} will be very useful to our purposes. 

\subsection{Local existence and continuation lemmas}\label{loc-exi}

The existence and uniqueness of a local positive solution to \eqref{Cauchy pb} is rather classical, but for the sake of completeness we provide a full proof.

\begin{lemma}\label{lem: lem}
Let $p,q>0$. Then for every $(\xi,\eta) \in (0,+\infty)^2$ there exists $\rho \equiv  \rho(\xi, \eta)>0$, depending continuously on $(\xi,\eta)$, such that problem \eqref{Cauchy pb} has a unique positive solution for $r \in (0,\rho)$. 
\end{lemma}
\begin{proof}
Let 
\[
X:= \left\{(u,v) \in C\!\left([0,\rho] ; \R^2\right): \  \left\| (u,v)-(\xi,\eta) \right\|_\infty \le \frac12 \, \xi \wedge \eta  \right\},
\]
where $\rho$ has to be chosen later, and $\| (\cdot,\cdot) \|_\infty$ denotes the norm obtained as the maximum between the usual $ L^\infty $ norms of the two components on $[0,\rho]$. Note that, by construction, since $\xi,\eta>0 $ the set $ X $ is formed by positive functions. Let also $F \equiv  (F_1,F_2):X \to C\!\left([0,\rho] ; \R^2\right)$ be defined by 
\[
\begin{aligned}
F_1(u,v)(r) & := \xi- \int_0^r \frac{1}{\psi^{n-1}(s)} \left( \int_0^s v^q \,  \psi^{n-1} \, dt \right) ds \, , \\
F_2(u,v)(r) & := \eta- \int_0^r\frac{1}{\psi^{n-1}(s)}  \left( \int_0^s u^p \, \psi^{n-1}  \, dt \right) ds \, . 
\end{aligned}
\]
It is not difficult to check that
\begin{equation}\label{eul1}
\begin{aligned}
\left|F_1(u,v)(r)-\xi\right| & \le \left(\frac32 \, \eta\right)^q \int_0^r \Theta \, ds \, ,  \\
\left|F_2(u,v)(r)-\eta\right| & \le \left(\frac32 \,  \xi\right)^p \int_0^r \Theta \, ds \, ,
\end{aligned}
\end{equation}
and 
\begin{equation}\label{eul2}
\begin{aligned}
\left|F_1(u_1,v_1)(r)- F_1(u_2,v_2)(r)\right| \le \left(\sup_{\tau \in \left[\frac12 \eta, \frac32 \eta \right]}  q \, \tau^{q-1} \right) \left\|v_1-v_2\right\|_\infty \int_0^r \Theta \, ds  \, , \\
\left|F_2(u_1,v_1)(r)- F_2(u_2,v_2)(r)\right| \le \left(\sup_{\tau \in \left[\frac12 \xi, \frac32 \xi \right]} p \, \tau^{p-1}\right) \left\|u_1-u_2\right\|_\infty \int_0^r \Theta \,ds  \, ,
\end{aligned}
\end{equation}
for every $r \in [0,\rho]$ and every $ (u_1,v_1),(u_2,v_2) \in X $. Since $\Theta$ is locally bounded (it is smooth on $ (0,+\infty) $ and behaves like $ r^2 $ near $r=0$), from \eqref{eul1} and \eqref{eul2} it follows easily that there exists $\rho>0$, depending on $\psi$ and $n$, and in a continuous fashion on $(\xi,\eta)$, such that $F$ is a contraction mapping from $X$ into itself. Thus, by the Banach fixed-point theorem, $F$ has a unique fixed point in $X$, which is clearly a positive solution  of \eqref{Cauchy pb} in $ (0,\rho) $.  
\end{proof}

The local solution constructed above is always positive for small $r$, and by classical ODE theory can be extended to a maximal interval $[0,T)$, possibly changing sign. Note that, if $ p \wedge q \ge  1 $, the extension and its lifetime $ T>0 $ only depend on the initial data $ (\xi,\eta) $. However, if $ p \wedge q <  1 $, the uniqueness of the solution may fail past any point $\bar r>0$ where either $u(\bar r)=0$ or $v(\bar r)=0$; thus, the value of $T$ can in this case also depend on the specific chosen extension, and not only on $ (\xi,\eta) $. Nevertheless, we will see that it is still possible to obtain a quantitative lower bound on $T$. 

Let us now set 
\beq\label{def R}
R_{\xi,\eta}
:= 
\sup \left\{r \in [0,T): \ u(t) \wedge v(t) > 0 \quad \forall t \in (0,r)\right\} ,
\eeq
namely the size of the maximal positivity interval of the solution. Note that, by definition, if $ R_{\xi,\eta} < +\infty  $ either $u\!\left(R_{\xi,\eta}\right)=0$ or $v\!\left(R_{\xi,\eta}\right)=0$, and $R_{\xi,\eta}$ is uniquely determined by $\xi$ and $\eta$ (as opposed to $ T $), since the solution is unique as long as it stays positive. Moreover, still in the case where $R_{\xi,\eta} $ is finite, the continuation theorem for ODE ensures that $T>R_{\xi,\eta}$, and one between $u$ and $v$ must necessarily change sign past $r=R_{\xi,\eta}$, due to the strong maximum principle. 

As a straightforward consequence of the definition of $ R_{\xi,\eta} $ and \eqref{eq int form}, we have the following fundamental monotonicity result, along with the characterization of the limits at infinity in the stochastically complete case. 

\begin{lemma}\label{lem: mon}
Let $ p,q>0 $ and $(\xi,\eta) \in (0,+\infty)^2$. Let $(u,v)$ solve \eqref{Cauchy pb}. Then 
\[
u'(r) <0 \quad \text{and} \quad v'(r)<0 \qquad \forall r \in \left(0 , R_{\xi,\eta} \right) . 
\]
In particular, if $ R_{\xi,\eta}=+\infty $ (namely $ (u,v) $ is a globally positive solution) there exist finite the limits
\[
\ell_u := \lim_{r \to +\infty} u(r) \ge 0  \, , \qquad \ell_v := \lim_{r \to +\infty} v(r) \ge 0\, .
\]
\end{lemma}

\begin{corollary}\label{prop: sc to 0} 
	Let $ \Theta \not \in L^1(\R^+) $. Let $ p,q>0 $ and $(\xi,\eta) \in (0,+\infty)^2$. Then, if $(u,v)$ is a globally positive solution to \eqref{Cauchy pb}, it satisfies
	$$
	\ell_u = \ell_v = 0 \, .
	$$
\end{corollary}
 \begin{proof}
 	Suppose by contradiction that one of these limits, say $\ell_u$, is strictly positive (if $\ell_u=0$ and $\ell_v>0$ the argument can be repeated in the same way). Upon integrating \eqref{eq int form} on $(r,+\infty)$, we deduce that 
 	\[
 	v(r) - \ell_v = \int_r^{+\infty} \frac{1}{\psi^{n-1}(s)} \left(\int_0^s u^p \, \psi^{n-1}  \, dt \right) ds  \qquad \forall r>0 \, .
 	\]
 	The above integral can be easily estimated by recalling that $u$ is monotone decreasing due to Lemma \ref{lem: mon}, hence 
 	\[
 	v(r) - \ell_v \ge \ell_u^p \int_r^{+\infty} \Theta \, ds = +\infty \qquad \forall r>0 \, ,
 	\]
 	which is clearly not possible.	
 \end{proof}
 
In the proof of Proposition \ref{thm: main ex}, that is the crucial fact that under \eqref{crit-sup} for every $ \xi>0 $ there exists at least one $ \eta \equiv \eta(\xi) > 0 $ that gives rise to a globally positive solution, we will need the continuous dependence of (suitable extensions of) the solutions to \eqref{Cauchy pb} with respect to $(\xi, \eta)$ also \emph{beyond} the positivity radius $R_{\xi,\eta}$. Again, this follows from standard ODE theory if $  p \wedge q \ge  1 $, but since \eqref{crit-sup} allows one exponent to be strictly smaller than $1$ (at least in dimension $n \ge 5$), we need an argument which covers all these cases. 

\begin{lemma}\label{lem: cont}
Let $ p,q>0 $ and $(\xi,\eta) \in (0,+\infty)^2$. Let $ (\bar u, \bar v) $ be the solution to \eqref{Cauchy pb} provided by Lemma \ref{lem: lem}. Let us fix any $\sigma \in \left(\rho, R_{\xi,\eta}\right)$. Then there exists an extension of $(\bar u, \bar v)$ whose maximal existence interval contains $[0, \sigma + \beta]$, where 
\beq\label{def b}
\beta := \min\left\{\frac{\sigma}{2^q} \left(-1+\sqrt{1+\frac{C_1}{\sigma^2}}\right) , \, \frac{\sigma}{2^p} \left(-1+\sqrt{1+\frac{C_2}{\sigma^2}}\right)\right\}>0 \, ,
\eeq
for some positive constants $C_1,C_2>0$ depending continuously on $(\xi,\eta)$ and independent of $\sigma$.
\end{lemma}
\begin{proof}
We let $Y$ denote the subset of $C\!\left([\sigma,\sigma+\beta];\R^2\right)$ consisting of all functions $ (u,v) $ which satisfy 
\beq\label{YYY}
 \left\|(u,v)-(\bar u,\bar v)(\sigma)\right\|_\infty \le \xi \wedge \eta \, , 
\eeq
with $ \sigma $ as in the statement and $\beta>0$ to be chosen later. By what we recalled above, we can assume that $ (\bar u, \bar v)  $ exists up to $ r=\sigma $. Let also $F \equiv  (F_1,F_2):Y \to C\!\left([\sigma,\sigma+\beta] ; \R^2\right)$ be defined by 
\[
\begin{aligned}
F_1(u,v)(r) & := \bar u(\sigma) + \psi^{n-1}(\sigma) \, \bar u'(\sigma) \int_\sigma^r \frac{1}{\psi^{n-1}} \,ds - \int_\sigma^r \frac{1}{\psi^{n-1}(s)} \left( \int_\sigma^s |v|^{q-1} v \, \psi^{n-1}  \, dt \right) ds \, , \\
F_2(u,v)(r) & := \bar v(\sigma) + \psi^{n-1}(\sigma) \, \bar v'(\sigma) \int_\sigma^r \frac{1}{\psi^{n-1}} \,ds - \int_\sigma^r \frac{1}{\psi^{n-1}(s)} \left( \int_\sigma^s |u|^{p-1} u \, \psi^{n-1} \, dt \right) ds \, . 
\end{aligned}
\]
The set $Y$ is clearly closed and convex, and it is easy to check that $F$ is a  continuous mapping from $ Y $ to $C\!\left([\sigma,\sigma+\beta] ; \R^2\right)$,  since $p,q >0$. We aim at showing that $F$ has a precompact image, and that for $\beta>0$ conveniently chosen it maps $Y$ into itself. As a result, the Schauder fixed-point theorem (see \emph{e.g.}~\cite[Corollary 11.2]{GT}) will ensure the existence of a fixed point $( \hat u, \hat v)$ for $F$ in $Y$, which is nothing but a solution of the Cauchy problem 
\[
\begin{cases}
\left(\psi^{n-1}\,\hat u'\right)' + \psi^{n-1} \, |\hat v|^{q-1}  \hat v =0 & \text{ for $r\in (\sigma,\sigma+\beta)$}  \\
\left(\psi^{n-1}\,\hat v'\right)' + \psi^{n-1} \, |\hat u|^{p-1} \hat u =0 & \text{ for $r\in (\sigma,\sigma+\beta)$}  \\
\hat u'(\sigma) = \bar u'(\sigma) \, , \quad \hat v'(\sigma) = \bar v'(\sigma) \, , \\
\hat u(\sigma) = \bar u(\sigma) \, , \quad \hat v(\sigma) = \bar v(\sigma) \, .
\end{cases}
\]
Therefore, because $ \bar{u}(\sigma),\bar{v}(\sigma) > 0 $ and the solution is unique as long as it is positive, we can assert that $ (\hat{u},\hat{v}) $ is the desired extension. Let us then prove that the assumptions of the Schauder fixed-point theorem are met. 

We start by showing that $F$ has a precompact image. Since $\sigma < R_{\xi,\eta} $, by Lemma \ref{lem: mon} we have that $0<\bar u(\sigma)<\xi$ and $0<\bar v(\sigma)<\eta$. Hence, if $ (u,v) \in Y $, it follows that $\| u \|_\infty \le 2 \xi$ and $ \| v \|_\infty \le 2 \eta $. Now, by the definition of $F$, for every $\sigma \le r_0 < r \le \sigma + \beta$ it holds 
\[
\begin{aligned}
\left|F_1(u,v)(r)-F_1(u,v)(r_0)\right| & \le \psi^{n-1}(\sigma) \left| \bar u'(\sigma) \right| \int_{r_0}^r \frac{1}{\psi^{n-1}} \,ds + \int_{r_0}^r \frac{1}{\psi^{n-1}(s)}  \left(\int_\sigma^s |v|^{q} \, \psi^{n-1} \, dt \right) ds \\
& \le C \int_{r_0}^r \frac{1}{\psi^{n-1}} \, ds  + \left(2\eta\right)^q \int_{r_0}^r \Theta \, ds \, ,
\end{aligned}
\]
where $C$ is a positive constant depending on $ \sigma , \xi , \eta , n $ through $ \psi $ and $ \bar{u} $. A similar expression can be derived for $\left|F_2(u,v)(r)-F_2(u,v)(r_0)\right|$ and, recalling the regularity and positivity properties of $\psi$ (along with the fact that $\sigma > \rho>0$), we readily deduce the equicontinuity of $  F(Y) $. Concerning the (quantitative) uniform boundedness, note that by \eqref{eq int form} we have
\[
-\psi^{n-1}(\sigma) \, \bar u'(\sigma) = \int_0^\sigma \bar v^q \, \psi^{n-1}  \, ds \, .
\]
Therefore, thanks to the monotonicity of $\psi$, it follows that 
\beq\label{2391}
\begin{split}
\left| F_1(u,v)(r)-\bar u(\sigma) \right| & \le \psi^{n-1}(\sigma) \left| \bar u'(\sigma) \right| \int_{\sigma}^r \frac{1}{\psi^{n-1}} \,ds + \int_{\sigma}^r \frac{1}{\psi^{n-1}(s)}  \left(\int_\sigma^s |v|^{q} \, \psi^{n-1} \, dt \right) ds \\
& \le \left(\eta^q \int_0^\sigma \psi^{n-1} \,ds\right) \int_{\sigma}^{\sigma+\beta} \frac{1}{ \psi^{n-1}}\,ds  + \left(2\eta\right)^q \int_\sigma^{\sigma+\beta} \frac{1}{\psi^{n-1}(s)} \left(\int_{\sigma}^s \psi^{n-1}\,dt\right)ds \\
& \le \eta^q \sigma \beta + \left(2\eta\right)^q \frac{\beta^2}{2} \, , 
\end{split}
\eeq
for all $ r \in [\sigma,\sigma+\beta] $. Similarly, we obtain
\beq\label{2392}
\left|F_2(u,v)(r)-\bar v(\sigma)\right| \le \xi^p \sigma \beta + \left(2\xi\right)^p \frac{\beta^2}{2} \, .
\eeq 
This proves the $F(Y)$ is also bounded and hence, by the Ascoli-Arzel\`a theorem, it is indeed a precompact subset of $ C\!\left([\sigma,\sigma+\beta] ; \R^2\right) $. 

It remains to show that $F: Y \to Y $ for suitable choice of $\beta>0$. By \eqref{2391} and \eqref{2392}, this is the case if
\[
\left[ \eta^q \sigma \beta + \left(2\eta\right)^q \frac{\beta^2}{2} \right] \vee  \left[ \xi^p \sigma \beta + \left(2\xi\right)^p \frac{\beta^2}{2} \right] \le  \xi \wedge \eta \, .
\]
It is straightforward to verify that such inequality holds for $\beta$ as in \eqref{def b}, provided 
\[
C_1 = \frac{2^{q+1} \, \xi \wedge \eta}{\eta^q} \, , \qquad  C_2 = \frac{2^{p+1} \, \xi \wedge \eta}{\xi^p} \, .
\] 
Note, in particular, that $C_1$ and $C_2$ depend continuously on $(\xi,\eta)$ and are independent of $\sigma$. The proof is thus complete.
\end{proof}

If $R_{\xi,\eta}<+\infty$, by taking limits as $\sigma \to R_{\xi,\eta}^-$ in \eqref{def b} we deduce that the constructed solution to \eqref{Cauchy pb} is defined at least on the interval 
\[
\left[0,  R_{\xi,\eta} + \min\left\{\frac{R_{\xi,\eta}}{2^q} \left(-1+\sqrt{1+\frac{C_1}{R_{\xi,\eta}^2}}\right), \, \frac{R_{\xi,\eta}}{2^p} \left(-1+\sqrt{1+\frac{C_2}{R_{\xi,\eta}^2}}\right)\right\}\right] \supsetneq \left[0,R_{\xi,\eta}\right].
\] 
As already mentioned, such solution must change sign past $ R_{\xi,\eta} $, and it may not be unique beyond this threshold if $ p \wedge q <1$.

\subsection{Some fundamental identities and inequalities}
Given a solution to \eqref{Cauchy pb}, we introduce the associated energy function
$$
F_{(u,v)}(r) := u'(r) v'(r) + \frac{1}{p+1} \left|u(r)\right|^{p+1}  + \frac1{q+1} \left|v(r)\right|^{q+1} ,
$$
along with the \emph{Pohozaev function}
$$
P_{(u,v)}(r) := \left( \int_0^r \psi^{n-1} \, ds \right) F_{(u,v)}(r) + \psi^{n-1}(r) \left( \frac{u(r) v'(r)}{p+1} + \frac{u'(r) v(r)}{q+1}\right).  
$$

\begin{proposition}\label{prop: der P}
Let $ p,q>0 $ and $(\xi,\eta) \in (0,+\infty)^2$. Let $(u,v)$ solve \eqref{Cauchy pb}, with maximal existence interval $[0,T) $. Then 
\beq\label{deriv-energy}
F_{(u,v)}'(r) = - 2(n-1) \, \frac{\psi'(r)}{\psi(r)} \, u'(r) v'(r) 
\eeq
and 
\beq\label{deriv-poho}
P_{(u,v)}'(r) = K(r) u'(r) v'(r)
\eeq
for all $ r \in (0,T) $, where 
\[
K(r) := \left(\frac1{p+1}+\frac1{q+1}-\frac{n-2}{n}\right) \psi^{n-1}(r)- \frac{2(n-1)}{n} \left( \int_0^r \frac{\psi^n \, \psi''}{\left(\psi'\right)^2} \, ds\right) \frac{\psi'(r)}{\psi(r)} \, .
\]
\end{proposition}
\begin{proof}
Formula \eqref{deriv-energy} follows by direct calculations, using \eqref{radial eq}. Taking advantage of the latter and \eqref{eq int form}, we can then compute the derivative of $P_{(u,v)}$: 
\[
\begin{split}
P_{(u,v)}'(r) = & \, \psi^{n-1}(r) \, F_{(u,v)}(r) - 2(n-1)\left(\int_0^r \psi^{n-1} \, ds \right) \frac{\psi'(r)}{\psi(r)} \, u'(r) v'(r)  \\
& - \frac{\psi^{n-1}(r) \left|u(r) \right|^{p+1}}{p+1} -\frac{\psi^{n-1}(r) \left|v(r)\right|^{q+1}}{q+1} + \left( \frac{1}{p+1} + \frac{1}{q+1}\right) \psi^{n-1}(r)  u'(r) v'(r) \\
= & \left[ \left(1+  \frac{1}{p+1} + \frac{1}{q+1}\right) \psi^{n-1}(r) -2(n-1) \left( \int_0^r \psi^{n-1}\,ds \right) \frac{\psi'(r)}{\psi(r)} \right] u'(r) v'(r) \, .
\end{split}
\]
Integrating by parts, recalling that $ \psi(0)=0 $ and $ \psi'(0)=1 $, we have that
\[
\int_0^r \psi^{n-1} \, ds = \frac1{n} \int_0^r \frac{n \, \psi^{n-1} \, \psi'}{\psi'} \, ds = \frac{\psi^n(r)}{n \, \psi'(r)} + \frac 1 n \int_0^r \frac{\psi^n \, \psi''}{\left(\psi'\right)^2} \, ds \, ,
\]
and hence
\[
\begin{split}
P_{(u,v)}'(r)  & =  \left[\left(1+  \frac{1}{p+1} + \frac{1}{q+1}-2 \, \frac{n-1}{n}\right) \psi^{n-1}(r) - 2 \, \frac{n-1}{n} 
\left( \int_0^r \frac{\psi^n \, \psi''}{\left(\psi'\right)^2}\,ds\right) \frac{\psi'(r)}{\psi(r)}\right] u'(r) v'(r) \\
 & =  K(r) u'(r) v'(r) \, ,
\end{split}
\]
that is \eqref{deriv-poho}.
\end{proof}

The above result yields a key monotonicity property for $ P_{(u,v)} $, in the critical or supercritical case. 

\begin{proposition}\label{prop: poho mon}
Let $ p,q>0 $ fulfill \eqref{crit-sup}, and $(\xi,\eta) \in (0,+\infty)^2$. Let $(u,v)$ solve \eqref{Cauchy pb}. Then $K(r) \le 0$ for every $r > 0$, with $ K(r)=0 $ if and only if equality holds in \eqref{crit-sup} and  $\psi''(s)=0$ for every $s \in (0,r)$. In particular, we have that $P_{(u,v)}(r) \le 0$ for every $ r \in \left( 0 , R_{\xi,\eta} \right) $, with $ P_{(u,v)}(r) = 0 $ if and only if equality holds in \eqref{crit-sup} and $\psi''(s)=0$ for every $s \in (0,r)$.
\end{proposition}
\begin{proof}
The first part of the thesis is a direct consequence of the definition of $K$, since $\psi'' \ge 0$ everywhere on any Cartan-Hadamard model manifold. The second part follows from formula \eqref{deriv-poho} and the fact that $P_{(u,v)}(0) = 0$, recalling also Lemma \ref{lem: mon}. 
\end{proof}

We are now in position to prove our nonexistence result for \eqref{Cauchy pb} on balls. 

\begin{proof}[Proof of Proposition \ref{thm: non-ex bdd}]
 Assume, by contradiction, that the Cauchy problem \eqref{Cauchy pb} admits a solution $ (u,v) $ for suitable initial data $(\xi,\eta) \in (0,+\infty)^2 $, which is positive on $ (0,R) $ and satisfies $u(R) = v(R) = 0$ for some $R>0$. Then, by virtue of Lemma \ref{lem: mon}, we have that $u'(R)<0$ and $v'(R)<0$. However, since \eqref{crit-sup} holds, the definition of $ P_{(u,v)} $ and Proposition \ref{prop: poho mon} entail 
	\[
	0  \ge P_{(u, v)}(R) = \left(\int_0^{R} \psi^{n-1}\,ds\right) u'(R) v'( R) > 0 \, ,
	\]
	which is absurd. 
\end{proof}

As a consequence, we infer that in the critical or supercritical case either $ R_{\xi,\eta}=+\infty $ or $ R_{\xi,\eta}<+\infty $ and the two components of the solution \emph{do not vanish simultaneously} at $ r=R_{\xi,\eta} $. 

\medskip 

Finally, we show a fundamental ordering property for positive solutions.

\begin{lemma}\label{ordering} 
Let $ p,q>0 $. Let $ \xi_1 \ge \xi_2 >0 $ and $ \eta_2 > \eta_1>0 $. Then, if $ (u_1,v_1) $ and $ (u_2,v_2) $ are two positive solutions to \eqref{Cauchy pb} starting from $ (\xi_1,\eta_1) $ and $ (\xi_2,\eta_2) $, respectively, in the common interval $ (0,b) $ for some $ b \in (0,+\infty] $, the functions
\beqn\label{ordering-1-bis}
r \mapsto  u_1(r)-u_2(r)  \qquad \text{and} \qquad r \mapsto v_2(r)-v_1(r) 
\eeqn
are strictly increasing in $ (0,b) $. 
\end{lemma}
\begin{proof}
First of all, let us show that
\beq\label{ordering-1}
u_2(r) < u_1(r) \quad \text{and} \quad  v_2(r) > v_1(r) \qquad \forall r \in (0,b) \, .
\eeq
Note that it suffices to establish the right inequality only, as the validity of the latter implies (recall \eqref{eq int form}) 
\beq\label{ordering-2}
\begin{gathered}
u_2(r) = \xi_2 - \int_0^r \frac{1}{\psi^{n-1}(s)} \left( \int_0^s v_2^q \, \psi^{n-1} \, dt \right) ds < \xi_1 - \int_0^r \frac{1}{\psi^{n-1}(s)} \left( \int_0^s v_1^q \, \psi^{n-1} \, dt \right) ds = u_1(r) \\  \forall r \in (0,b) \, , 
\end{gathered}
\eeq
namely the left inequality. To this aim, we observe that by continuity $ v_2>v_1 $ at least in a small neighborhood of $ 0 $. Hence, if $ v_2>v_1 $ failed to hold in the whole $ (0,b) $ there would exist some $ a \in (0,b) $ such that $ v_2>v_1 $ in $ (0,a) $ and $ v_2(a) = v_1(a) $. However, by arguing exactly as in \eqref{ordering-2}, this would yield $ u_2 < u_1 $ in $ (0,a) $, which in turn entails
$$
v_2(a) = \eta_2 - \int_0^a \frac{1}{\psi^{n-1}(s)} \left( \int_0^s u_2^p \, \psi^{n-1} \, dt \right) ds >  \eta_1 - \int_0^a \frac{1}{\psi^{n-1}(s)} \left( \int_0^s u_1^p \, \psi^{n-1} \, dt \right) ds = v_1(a) \, , 
$$
a contradiction. Therefore, \eqref{ordering-1} holds, and since for all $ r \in (0,b) $ we have 
$$
\left( u_1-u_2 \right)' \!(r) = \frac{1}{\psi^{n-1}(r)} \int_0^r \left(  v_2^q - v_1^q \right)  \psi^{n-1} \, ds  \quad \, \text{and} \quad \, \left( v_2-v_1 \right)' \!(r) = \frac{1}{\psi^{n-1}(r)} \int_0^r \left(  u_1^p - u_2^p \right)  \psi^{n-1} \, ds \, ,
$$
such derivatives are strictly positive.  
\end{proof}

\section{Existence of (at least) one globally positive solution}\label{esistenza-base}

Our goal here is to establish an existence result for globally positive solutions, that covers both the stochastically complete and incomplete cases. To this aim, we carefully adapt the strategy developed by Serrin and Zou in \cite{SerZou}, and split the argument into some intermediate steps. 

\medskip 

Using the same notation as in Section \ref{sec: prel}, let us introduce the sets
\beqn\label{def A e B}
\begin{split}
	A &:= \left\{ (\xi,\eta) \in (0,+\infty)^2: \ \, R_{\xi, \eta}<+\infty \ \, \text{and} \ \, v\!\left(R_{\xi,\eta}\right)>u\!\left(R_{\xi,\eta}\right)=0\right\}, \\
	B &:= \left\{ (\xi,\eta) \in (0,+\infty)^2: \ \, R_{\xi, \eta}<+\infty \ \, \text{and} \ \, u\!\left(R_{\xi,\eta}\right)>v\!\left(R_{\xi,\eta}\right)=0\right\}.
\end{split}
\eeqn
Note that, in view of Proposition \ref{thm: non-ex bdd}, if \eqref{crit-sup} holds then $ A \cup B $ accounts for the whole set of initial data that \emph{do not} give rise to a globally positive solution of \eqref{Cauchy pb}. Since we will have to handle separately, in some parts of the proof, stochastically complete and incomplete manifolds, in the latter case we also define the quantity 
\beq\label{def small theta}
\theta := \int_0^{+\infty} \Theta \,dr \in (0,+\infty) \, ,
\eeq
where $ \Theta $ is the same function as in \eqref{def-theta}. 

\medskip 

In the next three lemmas we describe the main topological properties of $A$ and $B$.

\begin{lemma}\label{lem: A and B nonempty}
Let $ p,q>0 $. Then both $A \neq \emptyset$ and $B \neq \emptyset$. More precisely:\\
($i$) If $\Mb^n$ is stochastically complete, namely $\Theta \not \in L^1(\R^+)$, then 
\[
s,t >0 \, , \quad t > s^\frac{p+1}{q+1} \qquad \implies \qquad (\xi, \eta) \equiv (s, 2t) \in A \, ,
\]
and 
\[
s,t >0 \, , \quad t < s^\frac{p+1}{q+1} \qquad \implies \qquad (\xi, \eta) \equiv  (2s, t) \in B \, .
\] 
($ii$) If $\Mb^n$ is stochastically incomplete, namely $\Theta \in L^1(\R^+)$, then 
\[
s, t >0 \, , \quad t > \left( \theta s^p \right) \vee \left(\frac{s}{\theta}\right)^{\frac 1 q}  \qquad \implies \qquad (\xi, \eta) \equiv (s, 2t) \in A \, ,
\]
and
\[
s,t >0 \, , \quad s > \left( \theta t^q \right) \vee  \left(\frac{t}{\theta}\right)^{\frac 1 p} \qquad \implies \qquad (\xi, \eta) \equiv (2s, t) \in B \, .
\]
\end{lemma}
\begin{proof}
We prove only the statements regarding the set $A$, as those regarding the set $B$ are completely analogous. Let $s,t>0$. We consider the (local) solution to \eqref{Cauchy pb} with $(\xi,\eta) \equiv (s,2t)$, and define 
\[
I:= (0,s) \times (t, 2t) \, , \qquad  R_I := \sup  \left\{r \in \left(0,R_{\xi,\eta}\right) : \ \, (u(r),v(r)) \in I \right\} .
\]
By integrating \eqref{eq int form} and exploiting the monotonicity of the components, we obtain 
\beq\label{eq lemma 3}
u(r) -s \le - t^q \int_0^r \Theta \, d\sigma \qquad \text{and} \qquad  v(r)-2t \ge -s^p \int_0^r \Theta \, d\sigma 
\eeq
for all $r \in \left(0,R_I \right) $. If $R_I<+\infty$, then by continuity and again monotonicity either $u(R_I) = 0$ and $ v(R_I) \ge t  $ or $ u(R_I)>0 $ and $v(R_I) = t$. In the former case it is plain that $(s,2t) \in A$, so the proof is complete. Therefore, in what follows we aim at ruling out, under the stated assumptions on $ (s,t) $, both the possibilities $R_I =+\infty$ and $u(R_I)>0$ with $v(R_I)=t$. 

\medskip 

\noindent \emph{($i$)} $ \Theta \not \in L^1(\R^+) $. \\
If $R_I = +\infty$, then $(u,v)$ is a globally positive solution of \eqref{Cauchy pb}, so Lemma \ref{lem: mon} implies that $u(r) \to \ell_u \in [0,s)$ and $v(r) \to \ell_v \in [t, 2t)$ as $r \to +\infty$. However, this is in contradiction with Corollary \ref{prop: sc to 0}, since $t>0$. Thus $R_I <+\infty$. Suppose now that $u(R_I)>0$ and $v(R_I)=t$. As $\Theta \not \in L^1(\R^+)$ and $ \Theta>0 $, there exists $r_0 \in (0,+\infty)$ such that
\[
\int_0^{r} \Theta \, d\sigma - \frac{t}{s^p} \ \begin{cases} < 0 & \text{if $r<r_0$} \, , \\ = 0 & \text{if $r=r_0$} \, , \\ >0 & \text{if $r>r_0$} \, . \end{cases}
\]
By \eqref{eq lemma 3}, we have that
\[
t = v(R_I) \ge 2t -s^p \int_0^{R_I} \Theta \, dr \, ,
\]
whence $R_I \ge r_0$, and in particular $u(r_0) \ge u(R_I) >0$. On the other hand, still \eqref{eq lemma 3} entails
\[
u(r_0) \le s -t^q \int_0^{r_0} \Theta \, dr = s \left(1- \frac{t^{q+1}}{s^{p+1}}\right) < 0 \, ,
\]
where the last inequality follows from the assumptions on $(s,t)$. Therefore, we obtain a contradiction again. 

\medskip  

\noindent \emph{($ii$)} $ \Theta \in L^1(\R^+) $. \\ 
Recall that in this case our assumptions on $ (s,t) $ read $t> \theta s^p$ and $\theta t^q>s$. If $R_I = +\infty$, then by taking the limit as $ r \to +\infty $ in \eqref{eq lemma 3} we obtain 
\[
0 \le \ell_u \le s- \theta t^q  <0 \, , 
\]
a contradiction. Thus $ R_I < +\infty $. Suppose now that $u(R_I)>0$ and $v(R_I)=t$. Still by \eqref{eq lemma 3} and the positivity of $\Theta$, it holds
\[
t = v(R_I) \ge 2t - s^p \int_0^{R_I} \Theta \, dr  >   2t - \theta s^p \, ,
\]
whence it follows that $ \theta s^p > t$, which is absurd. 
\end{proof}

\begin{lemma}\label{open-sets}
Let $ p,q>0 $. Then both the sets $A$ and $B$ are open.
\end{lemma}
\begin{proof}
When $p \wedge q \ge1$, the result is essentially a consequence of the continuous dependence of the solutions of \eqref{Cauchy pb} with respect to the initial data $ (\xi,\eta) $. In order to deal with general exponents, so as to cover the non-Lipschitz case $ p \wedge q <1 $ as well, one can argue similarly to \cite[Section 5]{SerZou}, where the same issue was addressed in $ \R^n $. Since the modifications required to adapt their proof to our Riemannian setting are minor, we only sketch the main points of the strategy.

First of all, we fix an arbitrary $ (\xi_0 , \eta_0) \in A $ and let $ (u_0,v_0) $ denote the corresponding solution to \eqref{Cauchy pb} starting from $ (\xi_0, \eta_0) $. From the definition of $ A $, it follows that $ R_0 \equiv R_{\xi_0,\eta_0} < +\infty  $, $ u_0(R_0)=0 $ and $ v_0(R_0)>0 $, with $(u_0,v_0) $ positive in $ [0,R_0) $. The goal is to show that there exists $ \delta_0>0 $ (small enough) such that if $ \left| \xi - \xi_0  \right| \vee \left| \eta - \eta_0 \right| < \delta_0 $ then $ (\xi,\eta) \in A $, \emph{i.e.}~the solution $ (u,v) $ to \eqref{Cauchy pb} starting from $ (\xi, \eta) $ satisfies $ R_{\xi,\eta} < +\infty $, $ u(R_{\xi,\eta}) = 0 $ and $ v(R_{\xi,\eta}) > 0 $. To this aim, we proceed as follows.

\begin{enumerate}[i)]

\item There exist $ c>0 $ and $ \delta' \in (0,1) $ such that, if $  \left| \xi - \xi_0  \right| \vee \left| \eta - \eta_0 \right| < \delta' $, then the maximal existence interval of $ (u,v) $ contains at least the common interval $ [0,R_0+c] $. Moreover, the uniform bounds 
\beq \label{est-1}
\left| u(r) \right| \le 2 \xi \le 2(\xi_0+1) \quad \text{and} \quad \left| v(r) \right| \le 2 \eta \le 2(\eta_0+1)  \qquad \forall r \in [0,R_0+c]
\eeq
hold. This can be achieved by means of Lemma \ref{lem: cont}, taking advantage of the fact that the constants $ C_1,C_2 $ in \eqref{def b} depend continuously on $ (\xi,\eta) $, the functions belonging to the space $ Y $ comply with \eqref{YYY} and, in addition, continuous dependence of the solutions holds in every interval $ [0,S] $ for $ S<R_0 $ arbitrarily close to $ R_0 $, since $ u_0 , v_0 >0 $ in $ [0,S] $. In particular, for $ \delta' \equiv \delta'_S > 0  $ small enough, we have $ R_{\xi,\eta} >S $. 

\smallskip
\item There exist $ \varepsilon \in (0,c \wedge R_0 ) $ and $ \delta'' \in (0,\delta') $  such that, if $  \left| \xi - \xi_0  \right| \vee \left| \eta - \eta_0 \right| < \delta'' $, then 
\beq \label{deriv-est}
\left| u'(r) - u'_0(R_0) \right| \le \frac{1}{2} \left| u'_0(R_0) \right| \quad \text{and} \quad \left| v'(r) - v'_0(R_0) \right| \le \frac{1}{2} \left| v'_0(R_0) \right| \qquad \forall r \in [R_0-\varepsilon,R_0+\varepsilon] \, .
\eeq
As concerns the left estimate in \eqref{deriv-est} (one argues analogously for the right one), upon using \eqref{eq int form} and the triangle inequality it is not difficult to deduce that
\beq \label{deriv-est-bis}
\begin{aligned}
\left| u'(r) - u'_0(R_0) \right| \le & \left| u'(R_0-\varepsilon) - u'_0(R_0-\varepsilon) \right| + \left| u_0'(R_0) \right| \left| \frac{\psi^{n-1}(R_0)}{\psi^{n-1}(r)} - 1 \right| \\
&  + \frac{1}{\psi^{n-1}(r)} \int_{R_0-\varepsilon}^r v^q \, \psi^{n-1} \, ds  + \frac{1}{\psi^{n-1}(r)} \int_{R_0-\varepsilon}^{R_0} v_0^q \, \psi^{n-1} \, ds \, ,
\end{aligned}
\eeq 
for all $ r \in [R_0-\varepsilon,R_0+\varepsilon] $. Thanks to \eqref{est-1}, the last three terms on the right-hand side of \eqref{deriv-est-bis} can be made arbitrarily small by choosing $ \varepsilon $ small enough depending only on $ R_0,u_0'(R_0),\xi_0,\eta_0,q,\psi,n $. On the other hand, the first term can also be made arbitrarily small upon requiring  $  \left| \xi - \xi_0  \right| \vee \left| \eta - \eta_0 \right| < \delta'' $, for a suitable $ \delta'' \equiv \delta''_\varepsilon \in (0,\delta') $, still  due to continuous dependence (recall that $ (u_0,v_0) $ is positive in $ [0,R_0-\varepsilon] $).

\smallskip 
\item We notice that, in the whole interval $ [R_0-\varepsilon,R_0+\varepsilon] $, by virtue of \eqref{deriv-est} we have
$$
 u'(r) \le \frac{1}{2} u_0'(R_0) <  0 \qquad \text{and} \qquad \frac{3}{2} v_0'(R_0) \le v'(r) \le \frac{1}{2} v_0'(R_0) <  0 \, .
$$
In particular, upon integration, for every $ \alpha \in (0,\varepsilon) $ we obtain 
\beq \label{q1}
v(R_0+\varepsilon) \ge v(R_0-\alpha)-\frac{3}{2} \left| v_0'(R_0) \right|(\alpha+\varepsilon) \qquad \text{and} \qquad u(R_0+\varepsilon) \le u(R_0-\alpha) - \frac{1}{2} \left| u_0'(R_0) \right|(\alpha+\varepsilon) \, .
\eeq
With no loss of generality, we may further require $ \varepsilon $ and $ \alpha $ to be so small that 
\beq \label{q2}
\varepsilon \le \frac{v(R_0)}{6\left| v_0'(R_0) \right|} \qquad \text{and} \qquad 0 < u_0(R_0-\alpha) \le \frac{1}{4} \left| u_0'(R_0) \right| \varepsilon \, ,
\eeq
since $ v_0(R_0)>0 $ and $  u_0(R_0)=0 $. In view of  \eqref{q1}, \eqref{q2} and continuous dependence on $ [0,R_0-\alpha] $, we can choose $ \delta_0 \in (0,\delta'') $ so small that, if $  \left| \xi - \xi_0  \right| \vee \left| \eta - \eta_0 \right| < \delta_0 $, then $ (u,v) $ is positive in $ [0,R_0-\varepsilon]  $ with
$$
v(R_0+\varepsilon)  > 0 \qquad \text{and} \qquad u(R_0+\varepsilon) < 0 \, . 
$$  
Because both $ u $ and $ v $ are decreasing in $ [R_0-\varepsilon,R_0+\varepsilon] $, this implies that  $ R_{\xi,\eta} \in (R_0-\varepsilon,R_0+\varepsilon) $ and $ u(R_{\xi,\varepsilon}) = 0 $.  Hence $ (\xi,\eta) \in A $  as desired (the proof for $ B $ is similar).  

\end{enumerate}

\end{proof}

\begin{lemma}\label{lem: functions A and B}
Let $ p,q>0 $. Then there exist two functions $ \phi,\gamma $ which are continuous and strictly increasing bijections of $ (0,+\infty) $ into itself, such that
\[
\begin{gathered}
\left\{(\xi, \eta) \in (0,+\infty)^2 : \ \,  \xi < \phi(\eta)\right\}  \subset A \, , \\
\left\{(\xi, \eta) \in (0,+\infty)^2 : \ \, \eta < \gamma(\xi)\right\} \subset B \, .
\end{gathered}
\]
\end{lemma}
\begin{proof}
Let $ s> 0$, and define
\[
f(s) := \begin{cases} 
s^\frac{p+1}{q+1} & \mbox{if } \Theta \not \in L^1(\R^+) \, , \\ 
\left( \theta s^p \right) \vee \left(\frac{s}{\theta}\right)^{\frac 1 q} & \mbox{if } \Theta \in L^1(\R^+) \, . 
\end{cases}
\]
In both cases, it is clear that $f$ is a continuous and strictly increasing bijection of $(0,+\infty)$ into itself, with inverse function $ f^{-1} $. Moreover, by Lemma \ref{lem: A and B nonempty}, we have 
\[
s,t > 0 \, ,  \quad  t>f(s)  \qquad \implies \qquad (\xi, \eta) \equiv (s,2t) \in A \, .
\]
In particular, we deduce that  
\[
 \xi , \eta \in (0,+\infty)^2 \, , \quad \xi < f^{-1}\!\left(\frac{\eta}{2}\right)  \qquad \implies \qquad (\xi, \eta)  \in A  \, .
\]
This shows the part of the thesis regarding $A$, with
\[
\phi(\eta):=  f^{-1}\!\left(\frac{\eta}{2}\right)  \qquad \forall \eta>0 \, .
\]
The proof of the second part of the thesis, regarding the set $ B $, is completely analogous and therefore we omit it.
\end{proof}

We are finally in position to prove the most important result of this section.

\begin{proposition}\label{thm: main ex}
Let $ p,q>0 $ fulfill \eqref{crit-sup}. Then, for each $ \xi>0 $ there exists at least one value $ \eta \equiv \eta(\xi)>0 $ such that $ (\xi,\eta) $ gives rise to a globally positive solution to \eqref{Cauchy pb}. In particular, \eqref{LE sys} on $\Mb^n$ has at least a $1$-parameter family of solutions.
\end{proposition}
\begin{proof}
Let $ \xi>0 $. Thanks to Lemma \ref{lem: functions A and B}, we have that $ (\xi, \bar \eta) \in B  $ for every $ \bar \eta >0 $ such that $ \bar \eta  < \gamma(\xi) $. Moreover, still by Lemma \ref{lem: functions A and B}, it follows that $ (\xi, \hat \eta) \in A $ for every $ \hat \eta >0 $ such that $ \hat \eta  > \phi^{-1}(\xi) $. Thus, since $ A $ and $B$ are open (due to Lemma \ref{open-sets}) and disjoint sets, there necessarily exists $ \eta >0 $  (depending on $ \xi$) such that $ (\xi,\eta) \not \in A \cup B $. Let us consider the solution $ (u,v) $ to the Cauchy problem \eqref{Cauchy pb} with this initial datum. Then there are two possibilities: either $ R_{\xi,\eta}=+\infty $, which means that $ (u,v) $ is a globally positive solution as desired, or $ R_{\xi,\eta}<+\infty $, and in this case $ u\!\left(R_{\xi,\eta}\right)=v\!\left(R_{\xi,\eta}\right)=0 $. However, the latter cannot occur in view of Proposition \ref{thm: non-ex bdd}, so the proof is complete.
\end{proof}

\section{Structure of globally positive solutions: proof of Theorems \ref{monotone-coro} and \ref{teo-cs-interv}} \label{gps}

In this section we prove our main results concerning the structure, with respect to $ (\xi,\eta) $, of the existence region of globally positive solutions, along with their asymptotic behavior. In the stochastically complete case, the tools introduced above are enough, thus we will start from the proof of Theorem \ref{monotone-coro}. On the contrary, the situation is much more complicated in the stochastically incomplete case, as we will need several new preliminary results, hence we devote to it an entire subsection. 

\begin{proof}[Proof of Theorem \ref{monotone-coro}]	
	The fact that $ \ell_u=\ell_v=0 $, for any globally positive solution, has already been shown in Corollary \ref{prop: sc to 0}. Hence, let us focus on the rest of the statement. From Proposition \ref{thm: main ex}, we know that for each $ \xi>0 $ there exists at least one $ \eta >0 $ such that the solution to \eqref{Cauchy pb} is globally positive. Assume by contradiction that there exists another $ \bar \eta>0 $ such that also the solution to \eqref{Cauchy pb} with initial data $ (\xi,\bar \eta) $, which we denote by $ (\bar u , \bar v) $, is globally positive. With no loss of generality, we can suppose that $ \bar \eta > \eta $. Then, upon applying Lemma \ref{ordering} with $ \xi_1=\xi_2=\xi $, $  \eta_2=\bar \eta > \eta = \eta_1  $ and $ b=+\infty $, we would infer that 
	$$
 	\ell_{\bar v} - \ell_{v} = \lim_{r \to +\infty} \left[ \bar{v}(r) - v(r)  \right]  > \bar \eta - \eta >0 \, ,
	$$
	which is absurd still in view of Corollary \ref{prop: sc to 0}. Therefore, $ \eta \equiv \eta(\xi) $ is uniquely identified and it is well defined a function $ F : (0,+\infty) \to (0,+\infty) $ that to every $ \xi>0 $ associates such value. In order to show that it is nondecreasing, let $ \xi_1>\xi_2>0 $. If, by contradiction, $ F(\xi_1) < F(\xi_2)  $ then by reasoning as above we would obtain $ \ell_{v_2}-\ell_{v_1}>0 $, which is impossible. Now we observe that, by reversing the roles of $ p,q $ and $ \xi,\eta $ and repeating the above argument, one finds that it is well defined a function $ G : (0,+\infty) \to (0,+\infty) $ that to every $ \eta>0 $ associates the only value $ \xi \equiv \xi(\eta)>0 $ such that $ (\xi,\eta) $ gives rise to a globally positive solution to \eqref{Cauchy pb}. By construction, it is plain that $ G(F(\xi)) = \xi  $ and $ F(G(\eta))=\eta $ for all $ \xi,\eta>0 $, so $ F $ is a bijection of $ (0,+\infty) $ into itself with $ G=F^{-1} $. On the other hand, since $ F $ is nondecreasing, the only possibility is that it is actually strictly increasing and continuous. 
\end{proof}

\subsection{The stochastically incomplete case.}\label{SI}

First of all, note that if $\Mb^n$ is stochastically incomplete, then of course it cannot be isometric to $\R^n$. In terms of the function $\psi$, this means that $\psi''$ must be strictly positive somewhere. Therefore, in the light of Proposition \ref{prop: poho mon}, we deduce that in the critical or supercritical case $P_{(u,v)}(r) \le - C<0$ for every $r$ large enough. As a straightforward consequence,  we have that 
\[
\psi^{n-1}(r) \left[ u'(r) v(r)+u(r) v'(r)\right] \le -C \qquad \forall r \ge r_0 \, ,
\]
which in turn implies that 
\begin{equation}\label{eq-prod}
u(r)v(r) \ge C \int_r^{+\infty} \frac{1}{\psi^{n-1}} \, ds \qquad \forall r \ge r_0 \, ,
\end{equation}
for a suitable $ r_0>0 $ large enough. Note that the integral on the right-hand side is finite as a trivial consequence of the fact that $ \Theta \in L^1(\R^+) $. This basic estimate will be crucial in the proof of the next result, which is the cornerstone of this subsection. 

\begin{proposition}\label{not-both-vanish}
	Let $ \Theta \in L^1(\R^+) $ and $ p,q>0 $ fulfill \eqref{crit-sup}. Let $ (\xi,\eta) \in (0,+\infty)^2 $. Then, if $ (u,v) $ is a globally positive solution to \eqref{Cauchy pb}, it holds 
	\begin{equation*}\label{pp}
\lim_{r \to +\infty} u(r) \vee v(r) > 0  \, . 
	\end{equation*}
\end{proposition}
\begin{proof}
We argue by contradiction, assuming that both $ u(r) $ and $ v(r) $ vanish as $ r \to + \infty $. In such case, integrating first the differential equations in \eqref{Cauchy pb} from $r$ to $s$, and then integrating with respect to $s$ from $r$ to $+\infty$, we readily obtain the following identities:
\begin{equation}\label{eq-u}
0 = 1 + \frac{u'(r)}{u(r)} \, \psi^{n-1}(r) \int_r^{+\infty} \frac{1}{\psi^{n-1}} \, ds - \frac{1}{u(r)} \int_r^{+\infty} \frac{1}{\psi^{n-1}(s)} \left(\int_r^s v^q \, \psi^{n-1} \, dt\right)ds \, , 
\end{equation}
and 
\begin{equation}\label{eq-v}
0 = 1 + \frac{v'(r)}{v(r)} \, \psi^{n-1}(r) \int_r^{+\infty} \frac{1}{\psi^{n-1}} \, ds - \frac{1}{v(r)} \int_r^{+\infty} \frac{1}{\psi^{n-1}(s)} \left(\int_r^s u^p \, \psi^{n-1} \, dt\right)ds \, ,
\end{equation}
for all $r>0$. Note that \eqref{crit-sup} implies $ pq>1 $, hence we can and will assume with no loss of generality that $ p>1 $. Now we consider three possibilities: 
\begin{equation}\label{A}\tag{A}
\limsup_{r \to +\infty} \frac{v(r)}{u^p(r)} = + \infty \qquad \text{and} \qquad \liminf_{r \to +\infty} \frac{v(r)}{u^p(r)} = 0 \, , 
\end{equation}
\begin{equation}\label{B}\tag{B}
\liminf_{r \to +\infty} \frac{v(r)}{u^p(r)} > 0 \, ,
\end{equation}
\begin{equation}\label{C}\tag{C}
\limsup_{r \to +\infty} \frac{v(r)}{u^p(r)} < + \infty  \, .
\end{equation}
Since \eqref{A}, \eqref{B} and \eqref{C} cover all the scenarios, achieving a contradiction in each case will prove the thesis. 

Suppose that \eqref{A} holds; in particular, this entails the existence of a sequence $ r_m \to + \infty $ such that 
\begin{equation}\label{A-1}
\lim_{m \to \infty} \frac{v(r_m)}{u^p(r_m)} = 0 \qquad \text{and} \qquad \left( \frac{v}{u^p} \right)' \! (r_m) = 0 \quad \forall m \in \N \, .
\end{equation}
The rightmost identity is equivalent to
$$
v'(r_m) \, u^{-p}(r_m) - p \, v(r_m) \, u^{-p-1} (r_m)\, u'(r_m) = 0 \qquad \forall m \in \N  \, ,
$$
that is 
$$
\frac{1}{p} \, \frac{v'(r_m)}{v(r_m)} = \frac{u'(r_m)}{u(r_m)}  \qquad \forall m \in \N \, , 
$$
whence 
\begin{equation}\label{A-2}
\frac{u'(r_m)}{u(r_m)} \, \psi^{n-1}(r_m) \int_{r_m}^{+\infty} \frac{1}{\psi^{n-1}} \, ds = \frac{1}{p} \frac{v'(r_m)}{v(r_m)} \, \psi^{n-1}(r_m)\int_{r_m}^{+\infty} \frac{1}{\psi^{n-1}} \, ds  \ge -\frac{1}{p} \qquad \forall m \in \N \, , 
\end{equation}
where the inequality follows from \eqref{eq-v} evaluated at $ r=r_m $. On the other hand, the leftmost identity in \eqref{A-1} implies the existence of some constant $c>0$ such that 
\begin{equation}\label{A-1-bis}
v(r_m) \le c \, u^p(r_m) \qquad \forall m \in \N \, . 
\end{equation}
Hence, using \eqref{eq-u} with $ r=r_m $ we end up with 
\begin{equation}\label{A-3}
\begin{aligned}
0 = & \, 1 + \frac{u'(r_m)}{u(r_m)} \, \psi^{n-1}(r_m) \int_{r_m}^{+\infty} \frac{1}{\psi^{n-1}} \, ds - \frac{1}{u(r_m)} \int_{r_m}^{+\infty} \frac{1}{\psi^{n-1}(s)} \left( \int_{r_m}^s v^q \, \psi^{n-1} \, dt\right)ds \\
 \ge & \, \frac{p-1}{p} - \frac{v^q(r_m)}{u(r_m)} \int_{r_m}^{+\infty} \Theta \, ds \ge  \frac{p-1}{p} - c^q  u^{pq-1}  (r_m)\int_{r_m}^{+\infty} \Theta \,  ds
\end{aligned}
\end{equation}
for all $ m \in \N $, where we took advantage of \eqref{A-2}, \eqref{A-1-bis} and the fact that $v$ is decreasing. Since $ pq>1 $, $ p>1 $, and $\Theta \in L^1(\R^+)$, letting $ m \to \infty $ in \eqref{A-3} we reach the contradiction $ 0 \ge \frac{p-1}{p} $.

Suppose instead that \eqref{B} holds. This means that there exists a constant $ c>0 $ such that 
\begin{equation}\label{B-1}
u^p(r) \le c \, v(r) \qquad \forall r \ge 0 \, . 
\end{equation}
Let us rule out the possibility that
\begin{equation}\label{B-1-a}
\limsup_{r \to +\infty} \frac{v'(r)}{v(r)} \, \psi^{n-1} (r)\int_r^{+\infty} \frac{1}{\psi^{n-1}} \, ds \le -1 \, .
\end{equation}
Indeed, if \eqref{B-1-a} were satisfied, we could select $0 < \epsilon< \frac{1}{p+1}$ and $ r_\epsilon > r_0$ such that
$$
\frac{v'(r)}{v(r)} \, \psi^{n-1}(r) \int_r^{+\infty} \frac{1}{\psi^{n-1}} \, ds \le -(1-\epsilon) \qquad \forall r \ge r_\epsilon \, ,
$$
and a simple integration of this differential inequality on $(r_\epsilon,r)$ would yield
\begin{equation}\label{B-2}
v(r) \le C_\epsilon \left( \int_r^{+\infty} \frac{1}{\psi^{n-1}} \, ds\right)^{1-\epsilon} \qquad \forall r \ge r_\epsilon \, ,
\end{equation}
for some constant $C_\epsilon>0$. Hence, thanks to \eqref{eq-prod}, we would infer that 
\begin{equation}\label{B-3}
u(r) \ge \frac{C}{C_\epsilon} \left( \int_r^{+\infty} \frac{1}{\psi^{n-1}} \, ds\right)^{\epsilon} \qquad \forall r \ge r_\epsilon \, .
\end{equation}
But \eqref{B-2}, \eqref{B-3} and \eqref{B-1} are inconsistent, as $0<\epsilon<\frac{1}{p+1}$. Therefore, since \eqref{B-1-a} cannot hold, we deduce that there exist a sequence $ r_m \to +\infty $ and $ \alpha \in [0,1) $ such that 
\begin{equation}\label{B-1-b}
\lim_{m \to \infty} \frac{v'(r_m)}{v(r_m)} \, \psi^{n-1}(r_m) \int_{r_m}^{+\infty} \frac{1}{\psi^{n-1}} \, ds = -\alpha \, .
\end{equation}
Taking $ r=r_m $ in \eqref{eq-v}, using the fact that $ u $ is decreasing along with \eqref{B-1}, we obtain  
\begin{equation}\label{B-4}
\begin{aligned}
0 =  & \, 1 + \frac{v'(r_m)}{v(r_m)} \, \psi^{n-1}(r_m) \int_{r_m}^{+\infty} \frac{1}{\psi^{n-1}} \, ds - \frac{1}{v(r_m)} \int_{r_m}^{+\infty} \frac{1}{\psi^{n-1}(s)} \left(\int_{r_m}
^s u^p \, \psi^{n-1} \, dt\right)ds  \\
\ge & \, 1 + \frac{v'(r_m)}{v(r_m)} \, \psi^{n-1}(r_m) \int_{r_m}^{+\infty} \frac{1}{\psi^{n-1}} \, ds - \frac{u^p(r_m)}{v(r_m)} \int_{r_m}^{+\infty} \Theta \, ds \\
\ge & \,  1 + \frac{v'(r_m)}{v(r_m)} \, \psi^{n-1}(r_m) \int_{r_m}^{+\infty} \frac{1}{\psi^{n-1}} \, ds - c \, \int_{r_m}^{+\infty} \Theta \, ds
\end{aligned}
\end{equation}
for all $ m \in \N $. Hence, by passing to the limit in \eqref{B-4} as $ m \to \infty $, due to \eqref{B-1-b} we reach the contradiction $ 0 \ge 1-\alpha $.

Finally, suppose that \eqref{C} holds. This means that there exists a constant $c>0$ such that
\begin{equation}\label{C-1}
v(r) \le c \, u^p(r) \qquad \forall r \ge 0 \, . 
\end{equation}
Similarly to \eqref{B}, we can rule out the possibility that 
\begin{equation*}\label{C-2}
\limsup_{r \to +\infty} \frac{u'(r)}{u(r)} \, \psi^{n-1}(r) \int_r^{+\infty} \frac{1}{\psi^{n-1}} \, ds \le -1 \, , 
\end{equation*}
otherwise for any $ 0 < \epsilon < \frac{p}{p+1} $ and suitable constants $ r_\epsilon>r_0 $ and $ C_\epsilon>0 $ we would deduce the inequalities
$$
u(r) \le C_\epsilon \left( \int_r^{+\infty} \frac{1}{\psi^{n-1}} \, ds\right)^{1-\epsilon} \quad \text{and} \quad v(r) \ge \frac{C}{C_\epsilon} \left( \int_r^{+\infty} \frac{1}{\psi^{n-1}} \, ds\right)^{\epsilon} \qquad \forall r \ge r_\epsilon \, ,
$$
which are inconsistent with \eqref{C-1}. Hence, we can infer the existence of a sequence $ r_m \to +\infty $ and $ \beta \in [0,1) $ such that 
\begin{equation}\label{C-3}
\lim_{m \to \infty} \frac{u'(r_m)}{u(r_m)} \, \psi^{n-1} (r_m)\int_{r_m}^{+\infty} \frac{1}{\psi^{n-1}} \, ds = -\beta \, .
\end{equation}
Taking $ r=r_m $ in \eqref{eq-u}, by combining the decreasing monotonicity of $v$ and \eqref{C-1} we reach
\begin{equation*}\label{C-3-bis}
\begin{aligned}
0 =  & \, 1 + \frac{u'(r_m)}{u(r_m)} \, \psi^{n-1}(r_m) \int_{r_m}^{+\infty} \frac{1}{\psi^{n-1}} \, ds - \frac{1}{u(r_m)} \int_{r_m}^{+\infty} \frac{1}{\psi^{n-1} (s)} \left(\int_{r_m}
^s v^q \, \psi^{n-1} \, dt\right)ds  \\
\ge & \, 1 + \frac{u'(r_m)}{u(r_m)} \, \psi^{n-1}(r_m) \int_{r_m}^{+\infty} \frac{1}{\psi^{n-1}} \, ds - \frac{v^q(r_m)}{u(r_m)} \int_{r_m}^{+\infty} \Theta \, ds \\
\ge & \,  1 + \frac{u'(r_m)}{u(r_m)} \, \psi^{n-1} (r_m) \int_{r_m}^{+\infty} \frac{1}{\psi^{n-1}} \, ds - c^q u^{pq-1}(r_m) \, \int_{r_m}^{+\infty}  \Theta \, ds
\end{aligned}
\end{equation*}
for all $ m \in \N $, which leads again to the contradiction $ 0 \ge 1-\beta $ as $ m \to \infty $, since \eqref{C-3} holds and $ pq>1 $. The proof is thus complete.
\end{proof}

The above proposition motivates a better understanding of the asymptotic behavior of globally positive solutions when $ \Theta \in L^1(\R^+) $, as for the moment we only know that at least one of the two components has a strictly positive limit. This is the main purpose of the next intermediate results. 

\begin{lemma}\label{lem-interval}
Let $ p,q>0 $. Let $ \xi>0 $ and $ \eta_2 > \eta_1>0 $. Then, if $ (u_1,v_1) $ and $ (u_2,v_2) $ are two globally positive solutions to \eqref{Cauchy pb} starting from $ (\xi,\eta_1) $ and $ (\xi,\eta_2) $, respectively,  for each $ \eta \in (\eta_1,\eta_2) $ there exists a globally positive solution to \eqref{Cauchy pb} starting from $ (\xi,\eta) $. 
\end{lemma}
\begin{proof}
A locally positive solution $ (u,v) $ to \eqref{Cauchy pb} always exists by shooting, and it continues to exist as long as it is positive (recall the results of Subsection \ref{loc-exi}). So, let $ b \equiv R_{\xi,\eta}  \in (0,+\infty] $ be the largest number for which $ (u,v) $ is positive in the interval $ (0,b) $. If, by contradiction, $ b < +\infty $ then its definition would imply that either $ u(b)=0 $ or $ v(b)=0 $. However, due to Lemma \ref{ordering}, we can infer that 
$$
u_2(r) < u(r) < u_1(r) \quad \text{and} \quad v_1(r) < v(r) < v_2(r)  \qquad \forall r \in (0,b) \, ,
$$
which are clearly inconsistent with both $ u(b)=0 $ and $ v(b)=0 $, since $ (u_1,v_1) $ and $ (u_2,v_2) $ are globally positive solutions by assumption.   
\end{proof}

In the following, we let $ C_b\!\left([0,+\infty); \R^2 \right) $ denote the space of globally bounded and continuous functions on $ [0,+\infty) $ with values in $ \R^2 $. For notational convenience, we set $ | (a_1,a_2) | := |a_1| \vee |a_2|  $ for any $ a_1, a_2 \in \R $.  

\begin{lemma}\label{L1-compact}
	Let $ \Theta \in L^1(\R^+) $. Given any  $ C_1 ,C_2 >0 $, the set 
	\begin{equation}\label{compact}
	Z := \left\{ (u,v) \in C_b\!\left([0,+\infty) ; \R^2 \right) \Bigg| \begin{array}{l} \ \left\| (u,v) \right\|_\infty  \le C_1 \, , \\  \ \displaystyle \left| \left( u(r), v(r) \right) - \left( u(s) , v(s) \right) \right| \le C_2 \, \int_s^r \Theta \, dt \quad \forall r>s\ge 0 \end{array} \right\}
	\end{equation}
	is compact in $ C_b\!\left([0,+\infty); \R^2 \right)  $. 
\end{lemma}
\begin{proof}
Let $ \{ (u_k,v_k) \} \subset Z $. Since $ \Theta $ is locally bounded, from the definition of $ Z $ it follows that for every $ R>0 $ the sequence is uniformly bounded and uniformly Lipschitz in $ C\!\left( [0,R];\R^2 \right) $. Hence, by the Ascoli-Arzel\`a theorem and a standard diagonal procedure we can infer that there exist $ (\bar{u},\bar{v}) \in C\!\left([0,+\infty); \R^2 \right) $ and a subsequence $ \left\{ \left(u_{k_j},v_{k_j}\right) \right\} $ such that 
	\begin{equation}\label{loc-conv}
	\left(u_{k_j},v_{k_j}\right) \underset{j \to \infty}{\longrightarrow} (\bar{u},\bar{v})  \qquad \text{in } C\!\left( [0,R];\R^2 \right) \ \text{for every } R>0 \, . 
	\end{equation}
	In particular, pointwise convergence takes place, which readily ensures that $ (\bar{u},\bar{v}) \in Z $. We are therefore left with proving that convergence is uniform in the whole half line $ [0,+\infty) $. To this aim, for any arbitrary $ \varepsilon>0 $ we select $ R_\varepsilon > 0 $ so large that 
	\begin{equation}\label{loc-conv-int}
	\int_{R_\varepsilon}^{+\infty} \Theta \, dt < \frac{\varepsilon}{3C_2} \, .
	\end{equation}
	By virtue of \eqref{loc-conv}, there exists $ j_\varepsilon \in \N $ such that
	\begin{equation}\label{loc-conv-int-1}
	\left| \left(u_{k_j}(r),v_{k_j}(r)\right) - (\bar{u}(r),\bar{v}(r))  \right| < \frac{\varepsilon}{3} \qquad \forall r \in [0,R_\varepsilon] \, , \ \forall j > j_\varepsilon \, . 
	\end{equation}
	On the other hand, for larger values of $ r $ we have 
	\begin{equation}\label{loc-conv-int-2}
	\begin{aligned}
	 \big| \! \left( u_{k_j}(r), v_{k_j}(r) \right) & - \left(\bar{u}(r),\bar{v}(r)\right)\!  \big| & \\
	\le &  \left| \left(u_{k_j}(r),v_{k_j}(r)\right) - \left(u_{k_j}(R_\varepsilon),v_{k_j}(R_\varepsilon)\right) \right| + \left|  \left(u_{k_j}(R_\varepsilon),v_{k_j}(R_\varepsilon)\right) - \left(\bar{u}(R_\varepsilon),\bar{v}(R_\varepsilon)\right)  \right| \\
	 & + \left| \left(\bar{u}(R_\varepsilon),\bar{v}(R_\varepsilon)\right) - \left(\bar{u}(r),\bar{v}(r)\right) \right| \le 2C_2 \, \int_{R_\varepsilon}^r \Theta \, dt  +  \frac{\varepsilon}{3} < \varepsilon  \qquad \forall r > R_\varepsilon \, , \ \forall j > j_\varepsilon \, , 
	\end{aligned}
	\end{equation}
	where we exploited the definition of $ Z $ along with \eqref{loc-conv-int}. Taking the limit as $ j \to \infty $ in \eqref{loc-conv-int-1} and \eqref{loc-conv-int-2} we end up with
	$$
	\limsup_{j \to \infty} \sup_{r \in [0,+\infty)} \left| \left(u_{k_j}(r),v_{k_j}(r)\right) - \left(\bar{u}(r),\bar{v}(r)\right) \right| \le \varepsilon \, ,
	$$
	which completes the proof in view of the arbitrariness of $ \varepsilon $.
\end{proof}

\begin{lemma}\label{first-zero}
Let $ p,q>0 $. Given $ (\xi,\eta) \in (0,+\infty)^2 $, suppose that $ (u,v) $ is a globally positive solution to \eqref{Cauchy pb}. Let $ \left\{ \left( \xi_k,\eta_k \right) \right\} \subset (0,+\infty)^2 $ be a sequence that converges to $ (\xi,\eta) $. Let $ (u_k,v_k) $ denote the local solution to \eqref{Cauchy pb} starting from $ (\xi_k,\eta_k) $ and $ [0,R_k) $ its maximal positivity interval according to \eqref{def R}, with $ R_k \equiv R_{\xi_k,\eta_k} \in  (0,+\infty] $, for each $ k \in \N $. Then 
\beqn
\lim_{k \to \infty} R_k = +\infty \qquad \text{and} \qquad  (u_k,v_k) \underset{k \to \infty} \longrightarrow  (u,v) \quad \text{locally uniformly in } [0,+\infty) \, . 
\eeqn
\end{lemma}
\begin{proof}
With no loss of generality, we may assume that $ R_k < + \infty $ for all $ k \in \mathbb{N} $, so either $ u_k(R_k)=0 $ or $ v_k(R_k)=0 $. For simplicity, and up to subsequences, we discuss the former case only (if instead $ v_k(R_k)=0 $ one argues similarly). Hence, from a further integration of \eqref{eq int form} we obtain
\beq\label{zero-abs}
0 = \xi_k - \int_0^{R_k} \frac{1}{\psi^{n-1}(s)} \left( \int_0^s v_k^q \, \psi^{n-1} \, dt \right) ds \ge \xi_k - \eta_k^q \int_0^{R_k} \Theta \, ds \, ,
\eeq 
thus it is plain that $ \{ R_k \} $ stays bounded away from zero. Still up to subsequences, we may therefore assume, in addition, that $ R_k \to R $ as $ k \to \infty $ for some $ R \in (0,+\infty] $. Suppose by contradiction that $ R < +\infty $. As a result, for every $ S \in (0,R) $ the sequence $ \{ (u_k,v_k) \} $ is eventually positive and lies in a set of the form \eqref{compact}, up to replacing $ [0,+\infty) $ with $ [0,S] $. A local version of Lemma \ref{L1-compact} is therefore applicable and guarantees that, again up to subsequences, it converges in $ C\!\left([0,S] ; \R^2 \right) $ to a nonnegative function $ (\bar{u},\bar{v}) $, so by passing to the limit in the integral formulations satisfied by $ (u_k,v_k) $ we find that 
$$
\bar{u}(r) = \xi - \int_0^{r} \frac{1}{\psi^{n-1}(s)} \left( \int_0^s \bar{v}^q \, \psi^{n-1} \, dt \right) ds \qquad \forall r \in [0,S]
$$
and
$$
\bar{v}(r) = \eta - \int_0^{r} \frac{1}{\psi^{n-1}(s)} \left( \int_0^s \bar{u}^p \, \psi^{n-1} \, dt \right)
 ds \qquad \forall r \in [0,S] \, .
$$
However, this means that $ (\bar{u},\bar{v}) $ solves the same problem as $ (u,v) $ in $ [0,S] $, and therefore it must coincide with the latter in such interval (recall that $ (u,v) $ is globally positive by assumption). In particular, we can rewrite \eqref{zero-abs}, for large $k$, as 
\beqn
\begin{aligned}
0 = & \, \xi_k - \int_0^{R_k} \frac{1}{\psi^{n-1}(s)} \left( \int_0^s v_k^q \, \psi^{n-1} \, dt \right) ds \\
= & \, \xi_k - \int_0^{S} \frac{1}{\psi^{n-1}(s)} \left( \int_0^s v_k^q \, \psi^{n-1} \, dt \right) ds - \int_S^{R_k} \frac{1}{\psi^{n-1}(s)} \left( \int_0^s v_k^q \, \psi^{n-1} \, dt \right) ds  \\
\ge & \,  \xi_k - \int_0^{S} \frac{1}{\psi^{n-1}(s)} \left( \int_0^s v_k^q \, \psi^{n-1} \, dt \right) ds - \eta_k^q  \int_S^{R_k} \Theta  \, ds  \, ,
\end{aligned}
\eeqn 
whence, taking limits as $ k \to \infty $ and using the above established convergence, it follows that
\beqn
0 \ge \xi -  \int_0^{S} \frac{1}{\psi^{n-1}(s)} \left( \int_0^s v^q \, \psi^{n-1} \, dt \right) ds - \eta^q  \int_S^{R} \Theta \, ds =  u(S) - \eta^q  \int_S^{R}  \Theta \, ds \, .
\eeqn
Finally, letting $ S \uparrow R $, we would end up with
$$
0 \ge u(R) \, ,
$$
which is in contradiction with the global positivity of $u$. As a result, the only possibility is that $ R_k \to +\infty $ as $ k \to \infty $, and thus the previously shown uniform convergence holds (at least) locally in the whole $ [0,+\infty) $.  
\end{proof}

\begin{lemma} \label{glob-conv-unif} 
Let $ \Theta \in L^1(\R^+) $ and $ p,q>0 $. Let $ \left\{ \left( \xi_k,\eta_k \right) \right\} \subset (0,+\infty)^2 $ be a sequence that converges to some $ (\xi,\eta) \in (0,+\infty)^2 $, such that $ (u_k,v_k) $ is a globally positive solution to \eqref{Cauchy pb} starting from $ (\xi_k,\eta_k) $, for each $ k \in \N $. Then 
	\beqn
	(u_k,v_k) \underset{k \to \infty} \longrightarrow  (u,v) \qquad \text{uniformly in } [0,+\infty) \, ,
	\eeqn
	where $ (u,v) $ is a globally positive solution to \eqref{Cauchy pb} starting from $ (\xi,\eta) $.
\end{lemma}
\begin{proof}
It is enough to notice that, due to the positivity and the monotonicity of the components, the inequalities 
	$$
	0 \le u_k(r) \le \xi_k \quad \text{and} \quad 0 \le v_k (r)\le \eta_k  \qquad \forall r \in [0,+\infty)
	$$ 
	hold for every $ k \in \N $, thus it is readily seen that $ \{ (u_k,v_k) \} $ is contained in a set of the form \eqref{compact}. Hence, by virtue of Lemma \ref{L1-compact}, it admits a uniformly convergent subsequence $ \left\{ \left( u_{k_j} , v_{k_j} \right) \right\} $ to some $( u,v ) \in C_b\!\left( [0,+\infty) ; \R^2 \right) $, which is therefore also nonnegative. On the other hand, by passing to the limit in the integral identities
	$$
	u_k(r) = \xi_k - \int_0^r \frac{1}{\psi^{n-1}(s)} \left( \int_0^s v_k^q \, \psi^{n-1} \, dt \right) ds \, , \qquad v_k(r) = \eta_k - \int_0^r \frac{1}{\psi^{n-1}(s)} \left( \int_0^s u_k^p \, \psi^{n-1} \, dt \right) ds \, , 
	$$
	we infer that $ (u,v) $ actually solves \eqref{Cauchy pb} with initial datum $ (\xi,\eta) $, so it is identified as the globally positive solution starting from $ (\xi,\eta) $. Since the argument can be repeated along every subsequence of $ \{ (u_k,v_k) \} $, the claimed result holds for the whole sequence.
\end{proof}

\begin{lemma}\label{extendibility}
Let $ \Theta \in L^1(\R^+) $ and $ p,q>0 $. Given $ (\xi,\eta) \in (0,+\infty)^2 $, suppose that $ (u,v) $ is a globally positive solution to \eqref{Cauchy pb}. Then:

\begin{enumerate}[i)]

\item If $ \ell_u > 0 $, there exists $ \varepsilon>0 $ such that for every $ \bar{\eta} \in (\eta,\eta+\varepsilon) $ the solution to \eqref{Cauchy pb} starting from $ (\xi,\bar{\eta}) $ is globally positive;

\item If $ \ell_v > 0 $, there exists $ \varepsilon \in (0,\eta) $ such that for every $ \hat{\eta} \in (\eta-\varepsilon,\eta) $ the solution to \eqref{Cauchy pb} starting from $ (\xi,\hat{\eta}) $ is globally positive. 

\end{enumerate}

\end{lemma}
\begin{proof}
We proceed in both cases with an argument by contradiction. 

\begin{enumerate}[i)]

\item If the thesis were false, then there would exist a sequence $ \varepsilon_k \downarrow 0 $ such that the local solution $ (u_k,v_k) $ to \eqref{Cauchy pb}, starting from $ (\xi,\eta+\varepsilon_k) $, has a maximal positivity interval $ [0,R_k) $ with $ R_k \in (0,+\infty) $. Thanks to the comparison principle entailed by Lemma \ref{ordering}, it must necessarily be $ u_k $ that vanishes at $r=R_k$, so we can write 
\beq\label{lim-cont} 
\begin{aligned}
0 = & \, \xi - \int_0^{R} \frac{1}{\psi^{n-1}(s)} \left( \int_0^s v_k^q \, \psi^{n-1} \, dt \right) ds - \int_{R}^{R_k} \frac{1}{\psi^{n-1}(s)} \left( \int_0^s v_k^q \, \psi^{n-1} \, dt \right) ds \\
 \ge & \, \xi - \int_0^{R} \frac{1}{\psi^{n-1}(s)} \left( \int_0^s v_k^q \, \psi^{n-1} \, dt \right) ds - \left( \eta+\varepsilon_k \right)^q \int_{R}^{+\infty} \Theta \, ds \, ,
\end{aligned}
\eeq 
where $ R \in ( 0 , R_k ) $ is arbitrary but fixed for the moment. On the other hand, Lemma \ref{first-zero} ensures that $ R_k \to +\infty $ and $ \{ (u_k,v_k) \} $ converges locally uniformly to $ (u,v) $ as $ k \to \infty  $, thus we can pass to the limit in \eqref{lim-cont} to obtain 
\beqn\label{lim-cont-2}
0 \ge  \xi - \int_0^{R} \frac{1}{\psi^{n-1}(s)} \left( \int_0^s v^q \, \psi^{n-1} \, dt \right) ds - \eta^q \int_{R}^{+\infty}  \Theta \, ds =  u(R) - \eta^q \int_{R}^{+\infty}  \Theta \, ds \, .
\eeqn
Hence, by finally letting $ R \to +\infty $, we end up with the inequality
$$
0 \ge \ell_u \, ,
$$
which is absurd. 

\smallskip 

\item Similarly, to case i), denying the thesis would imply the existence of a sequence $ \varepsilon_k \downarrow 0 $ such that the local solution $ (u_k,v_k) $ to \eqref{Cauchy pb}, starting from $ (\xi,\eta-\varepsilon_k) $, has a maximal positivity interval $ [0,R_k) $ with $ R_k \in (0,+\infty) $. Still by virtue of Lemma \ref{ordering}, we deduce that in this case the component that vanishes at $ r=R_k $ is necessarily $ v $. Hence, as above we can write 
\beqn\label{lim-cont-3}
\begin{aligned}
0 = & \, \eta - \varepsilon_k - \int_0^{R} \frac{1}{\psi^{n-1}(s)} \left( \int_0^s u_k^p \, \psi^{n-1} \, dt \right) ds - \int_{R}^{R_k} \frac{1}{\psi^{n-1}(s)} \left( \int_0^s u_k^p \, \psi^{n-1} \, dt \right) ds \\
 \ge & \, \eta - \varepsilon_k - \int_0^{R} \frac{1}{\psi^{n-1}(s)} \left( \int_0^s u_k^p \, \psi^{n-1} \, dt \right) ds - \xi^p \int_{R}^{+\infty} \Theta \, ds 
\end{aligned} 
\eeqn 
for all $ R \in (0,R_k) $. Therefore, by using again Lemma \ref{first-zero} and passing to the limit as $ k \to \infty $, we obtain 
\beqn\label{lim-cont-4}
0 \ge  \eta - \int_0^{R} \frac{1}{\psi^{n-1}(s)}  \left( \int_0^s u^p \, \psi^{n-1} \, dt \right) ds - \xi^p \int_{R}^{+\infty}  \Theta \, ds  =  v(R) - \xi^p \int_{R}^{+\infty} \Theta \, ds \, ,
\eeqn
that is 
$$
0 \ge \ell_v
$$
upon taking the limit as $ R \to +\infty $, still a contradiction.  \qedhere
\end{enumerate}
\end{proof}

Before proving Theorem \ref{teo-cs-interv}, we establish a useful quantitative bound on the values of the limits at infinity $ \ell_u $ and $ \ell_v $. 

\begin{proposition}\label{abs-bound}
Let $ \Theta \in L^1(\R^+) $ and $ p,q>0 $ fulfill $ pq>1 $. Let $ (\xi,\eta) \in (0,+\infty)^2 $. Then, if $(u,v)$ is a globally positive solution to \eqref{Cauchy pb}, it holds
\beq\label{absb}
\ell_u \le \frac{ p^{\frac{q+1}{pq-1}} \left( q+1 \right)^{\frac{q+2}{pq-1}} }{ \left( p+1 \right)^{\frac{1}{pq-1}}  \left( pq-1 \right)^{\frac{q+1}{pq-1}} } \, \theta^{-\frac{q+1}{pq-1}}  \qquad \text{and} \qquad  \ell_v \le \frac{ q^{\frac{p+1}{pq-1}} \left( p+1 \right)^{\frac{p+2}{pq-1}} }{ \left( q+1 \right)^{\frac{1}{pq-1}}  \left( pq-1 \right)^{\frac{p+1}{pq-1}} } \, \theta^{-\frac{p+1}{pq-1}} \, , 
\eeq
where $ \theta $ was defined in \eqref{def small theta}. 
\end{proposition}
\begin{proof}
	First of all we observe that, by exploiting \eqref{eq int form} and using the monotonicity of the components, we easily obtain the inequalities 
	$$
	u'(r) \le -v^q(r)  \, \Theta(r)
	$$
	and
	$$
	v'(r) \le -u^p(r)  \, \Theta(r) \, ,
	$$
	for all $ r>0 $. Integrating the former from $ r $ to $ +\infty $, we infer that 
	$$
	\ell_u - u(r) \le - \int_r^{+\infty} v^q \, \Theta \, ds \qquad \Longrightarrow \qquad - u(r) \le - \int_r^{+\infty} v^q \, \Theta \, ds \, ,
	$$
	which substituted into the latter yields 
	\beq\label{int-grw}
	v'(r)  \le - \left( \int_r^{+\infty} v^q \, \Theta \, ds  \right)^p \Theta (r) \qquad \forall r>0 \, .
	\eeq
  Upon multiplying both sides of \eqref{int-grw} by $ v^q $, note that such inequality can be rewritten as   
  	\beq\label{int-grw-2}
  \frac{\left( v^{q+1} \right)'(r)}{q+1}  \le - \left( \int_r^{+\infty} v^q \, \Theta \, ds  \right)^p \, v^q(r) \, \Theta (r) \qquad \forall r>0 \, .
  \eeq 
  Hence, by further integrating \eqref{int-grw-2} from $ r $ to $ +\infty $ we end up with 
    	\beqn\label{int-grw-3}
  \frac{\ell_v^{q+1} - v^{q+1}(r)}{q+1}  \le - \frac{\left( \int_r^{+\infty} v^q \, \Theta \, ds  \right)^{p+1}}{p+1} \qquad \forall r>0 \, ,
  \eeqn
  which implies 
      	\beqn\label{int-grw-4}
  v(r)  \ge \left(\frac{q+1}{p+1}\right)^{\frac{1}{q+1}}  \left( \int_r^{+\infty} v^q \, \Theta \, ds  \right)^{\frac{p+1}{q+1}} \qquad \forall r>0 \, ,
  \eeqn
  and this inequality can equivalently be rewritten as 
        	\beq\label{int-grw-5}
  v^q(r) \, \Theta(r)  \left( \int_r^{+\infty} v^q \, \Theta \, ds  \right)^{-q \frac{p+1}{q+1}} \ge \left(\frac{q+1}{p+1}\right)^{\frac{q}{q+1}}  \Theta(r) \qquad \forall r>0 \, .
  \eeq 
  On the other hand, if we integrate \eqref{int-grw-5} from $ 0 $ to $r$, we find 
          	\beqn\label{int-grw-6}
 \frac{q+1}{pq-1} \left[ \left( \int_r^{+\infty} v^q \, \Theta \, ds  \right)^{-\frac{pq-1}{q+1}} - \left( \int_0^{+\infty} v^q \, \Theta \, ds  \right)^{-\frac{pq-1}{q+1}}  \right] \ge \left(\frac{q+1}{p+1}\right)^{\frac{q}{q+1}} \int_0^r \Theta \, ds \qquad \forall r>0 \, ,
  \eeqn
  that is, upon dropping the rightmost term on the left-hand side and rearranging factors,  
            	\beqn\label{int-grw-7}
  \left( \int_r^{+\infty} v^q \, \Theta \, ds  \right)^{\frac{pq-1}{q+1}}  \le   \frac{(q+1)^{\frac{1}{q+1}} \, (p+1)^{\frac{q}{q+1}}}{pq-1} \left( \int_0^r \Theta \, ds \right)^{-1} \qquad \forall r>0 \, .
  \eeqn
  Because $ v(r) > \ell_v $, the above estimate entails 
  $$
    \ell_v^{q \frac{pq-1}{q+1}} \left( \int_r^{+\infty} \Theta \, ds  \right)^{\frac{pq-1}{q+1}}  \le   \frac{(q+1)^{\frac{1}{q+1}} \, (p+1)^{\frac{q}{q+1}}}{pq-1} \left( \int_0^r \Theta \, ds \right)^{-1} \qquad \forall r>0 \, ,
  $$
  namely 
   \beq\label{int-grw-8}
  \ell_v  \le   \frac{(q+1)^{\frac{1}{q(pq-1)}} \, (p+1)^{\frac{1}{pq-1}}}{\left(pq-1\right)^{ \frac{q+1}{q(pq-1)} }} \left( \int_0^r \Theta \, ds \right)^{-\frac{q+1}{q(pq-1)}} \left( \theta - \int_0^r \Theta \, ds \right)^{-\frac{1}{q}}  \qquad \forall r>0 \, .
  \eeq  
  A straightforward optimization argument over $r$ ensures that the minimum of the right-hand side is attained if and only if 
  $$
  \int_0^r \Theta \, ds = \frac{q+1}{q(p+1)} \,  \theta \, , 
  $$
  so by substituting such value into \eqref{int-grw-8}, and carrying out some algebraic simplifications, we deduce the claimed bound on $ \ell_v $ in \eqref{absb}. The analogous bound on $ \ell_u $ is readily obtained by symmetry, i.e.~interchanging the roles of $ u$ and $ v $ along with those of $ p$ and $q $.
	\end{proof}

\begin{proof}[Proof of Theorem \ref{teo-cs-interv}]
Let $ \xi>0 $ be fixed. First of all we observe that, as a direct consequence of Lemma \ref{lem-interval}, the set of all $ \eta>0 $ for which there exists a globally positive solution to \eqref{Cauchy pb} is necessarily an interval, which we call $I$. Proposition \ref{thm: main ex} guarantees that $I$ is nonempty, and by virtue of Proposition \ref{not-both-vanish} and Lemma \ref{extendibility} we can also assert that $ I $ is not a singleton. Moreover, we claim that
\beq \label{el-eq-1}
\xi \ge \left( \eta - \theta \xi^p  \right)_+^q \theta  \qquad \text{and} \qquad \eta \ge \left( \xi - \theta \eta^q \right)_+^p \theta\,, 
\eeq
which readily ensure that $ I $ is in addition bounded and bounded away from zero. In order to obtain \eqref{el-eq-1} we notice that, by monotonicity, 
\beq\label{f1}
u(r) = \xi - \int_0^r \frac{1}{\psi^{n-1}(s)} \left(  \int_0^s v^q \, \psi^{n-1} \, dt \right) ds \ge \xi - \eta^q \int_0^r \Theta \, ds \ge \xi - \theta \eta^q 
\eeq
and
\beq\label{f2}
v(r) = \eta - \int_0^r \frac{1}{\psi^{n-1}(s)} \left(  \int_0^s u^p \, \psi^{n-1} \, dt \right) ds \ge \eta - \xi^p \int_0^r \Theta \, ds \ge \eta - \theta  \xi^p 
\eeq
for all $ r>0 $, so the desired bounds follow by plugging \eqref{f2} into \eqref{f1} and vice versa, using the positivity of the components and eventually letting $ r \to +\infty $. Hence, we can denote by $ \eta_m(\xi)$ and $\eta_M(\xi) $ the strictly positive and finite infimum and supremum of $I$, respectively. An immediate application of Lemma \ref{glob-conv-unif} shows that they are actually a minimum and a maximum, that is, both the pairs $ (\xi,\eta_m(\xi)) $ and $ (\xi,\eta_M(\xi)) $ give rise to globally positive solutions to \eqref{Cauchy pb}, \emph{i.e.}~$I$ is in addition closed and thus \eqref{prop-etamm-2} is necessary and sufficient for a global solution to exist. Still as a consequence of Proposition \ref{not-both-vanish}, Lemma \ref{extendibility} and the definitions of $ \eta_m $ and $ \eta_M $, it is plain that \eqref{upvz} and \eqref{uzvp} must hold. On the other hand, the validity of \eqref{upvp} follows from \eqref{upvz}, \eqref{uzvp} and Lemma \ref{ordering}: indeed, if $ (u_m,v_m) $ and $ (u_M,v_M) $ are the globally positive solutions to \eqref{Cauchy pb} starting from $ (\xi,\eta_m(\xi)) $ and $ (\xi,\eta_M(\xi)) $, respectively, and $ (u,v) $ is the one starting from $ (\xi,\eta) $, for any $ \eta \in \left( \eta_m(\xi) , \eta_M(\xi) \right) $, we have that
$$
\ell_{u}- \ell_{u_M} > \xi-\xi = 0 \qquad \text{and} \qquad \ell_{v}- \ell_{v_m} >\eta-\eta_m > 0 \, ,
$$
that is both $ \ell_u $ and $ \ell_v $ are strictly positive. 

Let us now prove the claimed properties of the functions $ \xi \mapsto \eta_m(\xi) $ and $ \xi \mapsto \eta_M(\xi) $, which by definition take values in $ (0,+\infty) $ and comply with \eqref{prop-etamm-bis}. To this aim, we can argue similarly to the proof of Theorem \ref{monotone-coro}. Assume by contradiction that there exist $ \xi_1 > \xi_2 > 0 $ such that $ \eta_m(\xi_1) < \eta_m(\xi_2) $, and let $ (u_1,v_1) $ and $ (u_2,v_2) $ denote the corresponding solutions to \eqref{Cauchy pb} starting from $ \left(\xi_1 , \eta_m(\xi_1)\right) $ and $ \left(\xi_2 , \eta_m(\xi_2)\right) $, respectively. Then, Lemma \ref{ordering} would entail 
$$
\ell_{v_2} - \ell_{v_1} > \eta_m(\xi_2)-\eta_m(\xi_1)>0 \, ,
$$
which is absurd since we already know that $ \ell_{v_2} = \ell_{v_1}=0 $. Hence, the function $ \xi \mapsto \eta_m(\xi)  $ is nondecreasing. Via an analogous argument, we infer that also $ \xi \mapsto \eta_M(\xi) $ is nondecreasing (in this case one has to use the monotonicity of $ u_1 - u_2 $). Upon reversing the roles of $ p,q $ and $ \xi,\eta $, we notice that it is well defined a function $  \xi_M : (0,+\infty) \to (0,+\infty) $ that to every $ \eta>0 $ associates the only value $ \xi \equiv \xi_M(\eta) > 0$ such that $ (\xi,\eta) $ gives rise to a globally positive solution to \eqref{Cauchy pb} satisfying $ \ell_v =0 $. By construction, we have that $  \xi_M\!\left( \eta_m(\xi) \right) = \xi $ and $  \eta_m\!\left( \xi_M(\eta) \right) = \eta $ for all $ \xi,\eta>0 $, which shows that $ \eta_m  $ is a bijection of $ (0,+\infty) $ into itself with $ \xi_M = \eta_m^{-1} $. Therefore, due to its monotonicity, it is necessarily strictly increasing and continuous. A completely analogous reasoning proves that the same properties hold for $ \eta_M $.

Finally, as for \eqref{prop-etamm-ter}, it is enough to observe that Lemma \ref{ordering} and Proposition \ref{abs-bound} yield
\beqn\label{bdd-int}
\eta_M(\xi) - \eta_m(\xi) < \ell_{v_M} - \ell_{v_m} \le C \qquad \forall \xi>0  \, ,
\eeqn
where $ C>0 $ is, for instance, the same constant appearing in the rightmost bound of formula \eqref{absb}. 
\end{proof}

\section{Rigidity of finite-energy solutions: proof of Theorem \ref{full-rigidity-proof}}\label{rigid}

Now, our goal is to show that globally positive solutions cannot have a finite energy, unless $  \Mb^n \equiv \R^n $ and $ q,p $ are critical. Before, we a need real-analysis lemma, which will be useful especially in the stochastically incomplete case.

\begin{lemma}\label{tl-state} 
	Let $ f \in C^1([1,+\infty)) $ satisfy $ f,f'>0 $ and $ \lim_{r\to+\infty} f(r) = +\infty $. Let $ \alpha , \epsilon > 0 $. Then  
	\beq\label{tl-2}
 \limsup_{r \to +\infty}   f^\epsilon(r) \int_r^{+\infty} \left( \frac{f}{f'} \right)^\alpha ds = + \infty \, . 
	\eeq
	\end{lemma}  
\begin{proof}
	First of all, we set
	$$
	g(t):= f^{-1}(t) \qquad \forall t \in \left[ f(1),+\infty \right)
	$$
	and apply the change of variables $ \tau =f(s) $ in the integral in formula \eqref{tl-2}, obtaining: 
	\beq\label{tl-3}
\int_r^{+\infty} \left( \frac{f}{f'} \right)^\alpha ds = \int_{f(r)}^{+\infty} \frac{\tau^\alpha}{\left[ f' \! \left( f^{-1}(\tau) \right) \right]^{\alpha+1}} \, d\tau = \int_{f(r)}^{+\infty} \tau^\alpha \left[ g' \! \left( \tau \right) \right]^{\alpha+1} d\tau \ge f^\alpha(r)  \int_{f(r)}^{+\infty} \left[ g' \! \left( \tau \right) \right]^{\alpha+1} d\tau \, . 
	\eeq 
	Assume by contradiction that \eqref{tl-2} does not hold, namely that there exists $ C>0 $ such that 
	$$
f^\epsilon(r) \int_r^{+\infty} \left( \frac{f}{f'} \right)^\alpha ds \le C \qquad \forall r \in [1,+\infty) \, .
	$$
	Upon setting $ r=f^{-1}(t) $ and taking advantage of \eqref{tl-3}, we would thus infer that 
	$$
 t^{\alpha+\epsilon} \int_{t}^{+\infty}  \left[ g' \! \left( \tau \right) \right]^{\alpha+1}  d\tau \le C  \qquad  \forall t \in \left[ f(1),+\infty \right) ,
	$$ 
	which in particular implies, by H\"older's inequality, 
		\beq\label{tl-4}
\int_t^{2t} g'(\tau) \, d\tau \le t^{\frac{\alpha}{\alpha+1}} \left( \int_{t}^{2t}  \left[ g' \! \left( \tau \right) \right]^{\alpha+1}  d\tau \right)^{\frac{1}{\alpha+1}} \le  C^{\frac{1}{\alpha+1}} t^{-\frac{\epsilon}{\alpha+1}} \qquad  \forall t \in \left[ f(1),+\infty \right) .
	\eeq 
	Finally, we apply \eqref{tl-4} with the choices $ t \equiv t_k:=f(1)\cdot 2^k $, for all $ k \in \mathbb{N} $. This yields
	$$
	\int_{t_k}^{t_{k+1}} g'(\tau) \, d\tau \le \frac{C^{\frac{1}{\alpha+1}}}{f(1)^{\frac{\epsilon}{\alpha+1}}} \, 2^{-\frac{\epsilon}{\alpha+1}k} \qquad \forall k \in \mathbb{N} \, ,
	$$
	whence, by adding up, 
	$$
	\int_{f(1)}^{+\infty} g'(\tau) \, d\tau \le  \frac{C^{\frac{1}{\alpha+1}}}{f(1)^{\frac{\epsilon}{\alpha+1}}} \cdot \frac{2^{\frac{\epsilon}{\alpha+1}}}{2^{\frac{\epsilon}{\alpha+1}}-1} \, ,
	$$
	that is $ g' \in L^1([f(1),+\infty)) $, which is inconsistent with the fact that $ g $ is surjective onto $ [1,+\infty) $.
\end{proof}

\begin{proof}[Proof of Theorem \ref{full-rigidity-proof}] 
Having in mind Remark \ref{R1}, in order to prove the thesis it is enough to establish that, as soon as 
\beq\label{nnn}
\Mb^n \not \equiv \R^n \qquad \text{or} \qquad  \frac{1}{p+1} + \frac{1}{q+1} < \frac{n-2}{n} \, ,
\eeq  
the globally positive solution $ (u,v) $ satisfies
\beq\label{nnn-2}
\int_0^{+\infty} u' v' \, \psi^{n-1} \, dr  = \int_0^{+\infty} u^{p+1} \, \psi^{n-1} \, dr = \int_0^{+\infty} v^{q+1} \, \psi^{n-1} \, dr = +\infty \, .
\eeq 
To this aim, as a consequence of Propositions \ref{prop: der P} and \ref{prop: poho mon}, it is readily seen that under \eqref{nnn} there exist constants $ r_0, K_0 >0 $ such that 
\beq\label{nnn-3}
\psi^{n-1}(r) \left( \frac{u(r) v'(r)}{p+1} + \frac{u'(r) v(r)}{q+1}\right) \le - K_0 \qquad \forall r \in [r_0,+\infty) \, ,
\eeq 
since requiring $ \Mb^n \not \equiv \R^n $ amounts asking that $ \psi''>0 $ in some interval. On the other hand, upon multiplying the first and the second equation in \eqref{Cauchy pb} by $v$ and $u$, respectively, and integrating by parts, we end up with the identities 
\beq\label{nnn-4}
\int_0^{r} u' v' \, \psi^{n-1} \, ds - \psi^{n-1}(r) u'(r)  v(r) =  \int_0^r v^{q+1} \, \psi^{n-1} \, ds \qquad \forall r > 0
\eeq 
and
\beq\label{nnn-4-bis}
\int_0^{r} u' v' \, \psi^{n-1} \, ds - \psi^{n-1}(r) u(r) v'(r) =  \int_0^r u^{p+1} \, \psi^{n-1} \, ds \qquad \forall r>0  \, .
\eeq 
In particular, recalling that $ u'<0 $ and $ v'<0 $, we easily obtain
$$
\int_0^{+\infty} u' v' \, \psi^{n-1} \, dr \le \int_0^{+\infty} v^{q+1} \, \psi^{n-1} \, dr \qquad \text{and} \qquad \int_0^{+\infty} u' v' \, \psi^{n-1} \, dr \le \int_0^{+\infty} u^{p+1} \, \psi^{n-1} \, dr \, ,
$$
which means that \eqref{nnn-2} is in fact equivalent to
\beq\label{nnn-5}
\int_0^{+\infty} u' v' \, \psi^{n-1} \, dr   = +\infty \, .
\eeq 
The proof that, under \eqref{nnn}, then \eqref{nnn-5} holds, will be our main focus from now on. To reach it, we will distinguish between the stochastically complete and incomplete case. 

\medskip 

\noindent i) Let $ \Theta \not \in L^1(\R^+) $. Thanks to Corollary \ref{prop: sc to 0}, we know that  
\beq\label{sc-1}
\lim_{r \to +\infty} u(r) = 0  \qquad \text{and} \qquad \lim_{r \to +\infty} v(r) = 0 \, . 
\eeq 
Assume by contradiction that 
\beq\label{sc-2}
\int_0^{+\infty} u' v' \, \psi^{n-1} \, dr < +\infty \, . 
\eeq 
As a consequence, by virtue of \eqref{nnn-4} and \eqref{nnn-4-bis} we can deduce that both the limits
$$
L_1 := \lim_{r \to +\infty} \psi^{n-1}(r) u'(r) v(r) \qquad \text{and} \qquad L_2 := \lim_{r \to +\infty} \psi^{n-1}(r) u(r) v'(r)
$$
exist, and they are clearly nonpositive. Moreover, upon letting $ r \to +\infty $ in \eqref{nnn-3}, we infer that  
$$
\frac{L_2}{p+1} + \frac{L_1}{q+1}  \le - K_0 \, ,
$$
which means that either $ L_2 < 0 $ or $ L_2=0 $ and $ L_1 < 0 $. In the former case, there exist constants $ r_1 ,K_1>0$ such that 
\beq\label{sc-5}
\psi^{n-1}(r) u(r) v'(r) \le -K_1 \qquad \forall r \in [r_1,+\infty) \, ,
\eeq
that is
$$
\psi^{n-1}(r) u'(r) v'(r) \ge  - K_1 \, \frac{u'(r)}{u(r)} \qquad \forall r \in [r_1,+\infty) \, .
$$ 
Upon integrating such inequality from $ r_1 $ to any $ r>r_1 $, we end up with 
$$
\int_{r_1}^r u' v' \, \psi^{n-1} \, ds  \ge K_1 \, \log\!\left( \frac{u(r_1)}{u(r)} \right)  \qquad \forall r \in [r_1,+\infty) \, ,
$$
which is clearly in contradiction with \eqref{sc-2}, recalling the left limit in \eqref{sc-1}. In the latter case one argues analogously, using the right limit in \eqref{sc-1} instead. 

\medskip 

\noindent ii) Let $ \Theta \in L^1(\R^+) $. From Proposition \ref{not-both-vanish}, we know that at least one between $ \ell_ u$ and $ \ell_v $ is strictly positive. If both $ \ell_u,\ell_v>0 $, then from \eqref{eq int form} and the monotonicity of $ u $ and $v $ it readily follows that 
\beq\label{nnn-8}
-u'(r) \ge \ell_v^q \, \frac{\int_0^r \psi^{n-1} \, ds}{\psi^{n-1}(r)} \quad \text{and} \quad -v'(r) \ge \ell_u^{p} \, \frac{\int_0^r \psi^{n-1} \, ds}{\psi^{n-1}(r)}  \qquad \forall r>0 \, ,
\eeq 
thus \eqref{nnn-5} holds provided we can show that 
\beq\label{nnn-6}
\int_0^{+\infty} \frac{\left( \int_0^r \psi^{n-1} \, ds \right)^2}{\psi^{n-1}(r)} \, dr   = +\infty \, .
\eeq 
This is in fact a simple consequence of Lemma \ref{tl-state} with the choices $ f(r) = \int_0^r \psi^{n-1} \, ds $ and $ \alpha=\epsilon=1 $, since  
$$
\int_0^{+\infty} \frac{\left( \int_0^r \psi^{n-1} \, ds \right)^2}{\psi^{n-1}(r)} \, dr  \ge \int_r^{+\infty} \frac{f^2}{f'} \, ds \ge f(r) \int_r^{+\infty} \frac{f}{f'} \, ds
$$
for all $ r \ge 1 $, whence \eqref{nnn-6} follows upon letting $ r \to +\infty $ along a sequence that attains the $ \limsup $ of the rightmost side. Let us therefore focus on the case where $ \ell_u = 0 $ and $ \ell_v > 0 $ (if instead $ \ell_v=0 $ and $ \ell_u>0 $ the argument is completely symmetric). The left inequality in \eqref{nnn-8} still holds, and its integration from $ r $ to $ +\infty $ yields 
\beq\label{nnn-7}
u(r) \ge \ell_v^q \, \int_r^{+\infty} \frac{\int_0^s \psi^{n-1} \, dt}{\psi^{n-1}(s)} \, ds  \qquad \forall r>0 \, .
\eeq 
On the other hand, by using this information along with the monotonicity of $u$ in the right identity of \eqref{eq int form}, we also deduce that 
\beq\label{nnn-7-bis}
-v'(r) \ge  u^{p}(r) \, \frac{\int_0^r \psi^{n-1} \, ds}{\psi^{n-1}(r)} \ge  \ell_v^{pq} \left( \int_r^{+\infty} \frac{\int_0^s \psi^{n-1} \, dt}{\psi^{n-1}(s)} \, ds \right)^p \frac{\int_0^r \psi^{n-1} \, ds }{\psi^{n-1}(r)}  \qquad \forall r>0 \, ,
\eeq 
so by multiplying \eqref{nnn-7} and \eqref{nnn-7-bis} we obtain
\beqn\label{nnn-11}
\psi^{n-1}(r) u(r) v'(r) \le - \ell_v^{(p+1)q} \left[ \left( \int_0^r \psi^{n-1} \, ds \right)^{\frac{1}{p+1}} \int_r^{+\infty} \frac{\int_0^s \psi^{n-1} \, dt}{\psi^{n-1}(s)} \, ds \right]^{p+1} \qquad \forall r>0 \, .
\eeqn
Thanks to Lemma \ref{tl-state} applied to the same $f$ as above, $ \alpha=1 $ and $ \epsilon = \frac{1}{p+1} $, we infer that the $\liminf $ of the right-hand side is $ -\infty $, hence also  
\beq\label{nnn-12}
\liminf_{r \to +\infty} \psi^{n-1}(r) u(r) v'(r) = -\infty \, .
\eeq
Suppose by contradiction that \eqref{sc-2} holds. Then, from \eqref{nnn-4-bis} we deduce again that the limit $ L_u $ exists, and in this case it must necessarily be equal $ -\infty $ due to \eqref{nnn-12}. This entails the validity of \eqref{sc-5} for other suitable constants $ r_1,K_1 > 0 $, which is however inconsistent with \eqref{sc-2} as shown in i) (recall that $ u $ vanishes at infinity by assumption).
\end{proof}

\section{Generalizations: proof of Corollary \ref{CH-remove}} \label{CG}

Finally, we show that all of our main results can be extended to a class of Riemannian models slightly wider than the Cartan-Hadamard one.

\begin{proof}[Proof of Corollary \ref{CH-remove}]
First of all, we observe that the preliminary results of Subsection \ref{loc-exi} hold regardless of the Cartan-Hadamard assumption. Moreover, a straightforward computation shows that requiring the function $ \mathcal{V} $ to be convex, that is $ \mathcal{V}'' \ge 0 $ on $ (0,+\infty) $, is equivalent to 
\beq \label{convex-poho}
 \left(1+  \frac{1}{p+1} + \frac{1}{q+1}\right) \psi^{n-1}(r) -2(n-1) \left( \int_0^r \psi^{n-1}\,ds \right) \frac{\psi'(r)}{\psi(r)} \le 0 \qquad \forall r \in (0,+\infty) \, .
\eeq
On the other hand, from the proof of Proposition \ref{prop: der P} it is clear that \eqref{convex-poho} is precisely what we need to assert that $P_{(u,v)}$ is monotone non-increasing, and $ P_{(u,v)}(r) \le 0 $ for every $ r \in (0,R_{\xi,\eta}) $, for any local solution to \eqref{Cauchy pb} starting from $ (\xi,\eta) \in (0,+\infty)^2 $. It is not difficult to verify that all the proofs of Section \ref{esistenza-base} and the proof of Theorem \ref{monotone-coro} solely rely on this property, \emph{i.e.}~we never use directly the fact that $ \psi'' \ge 0 $. As concerns Theorems \ref{teo-cs-interv} and \ref{full-rigidity-proof}, let us notice in addition that, if $ \Mb^n $ is a noncompact model manifold such that $ \mathcal{V} $ is convex and \eqref{nnn} holds, then $ \mathcal{V}''>0 $ somewhere. Indeed, if by contradiction $ \mathcal{V}''(r) = 0  $ for every $ r \in (0,+\infty) $, from the definition of $ \mathcal{V} $ we would end up with the identity 
$$
 \left( \int_0^r \psi^{n-1} \, ds \right)^{\frac{pq-1}{2(p+1)(q+1)}} = c r \qquad \forall r \in [0,+\infty)
$$  
for some $ c>0 $, that is
$$
\psi(r) =  \tilde{c} \, r^{\frac{1}{n-1} \left[ \frac{2(p+1)(q+1)}{pq-1} - 1 \right] } \qquad \forall r \in [0,+\infty)
$$
for another constant $ \tilde{c}>0 $. However, since $ \psi'(0)=1 $, this is possible if and only if $ \tilde{c} =1 $ and 
$$
\frac{1}{n-1} \left[ \frac{2(p+1)(q+1)}{pq-1} - 1 \right] = 1 \qquad \Longleftrightarrow \qquad \frac{1}{p+1} + \frac{1}{q+1} = \frac{n-2}{n} \, ,
$$
which also entails $ \psi(r)=r $, that is the exponents are critical and $ \Mb^n \equiv \R^n $, a contradiction. Hence, under \eqref{nnn} we can infer that $ \mathcal{V}'' $ must be positive somewhere, which implies in turn that inequality \eqref{convex-poho} is strict in an interval, so \eqref{nnn-3} does hold (recall Proposition \ref{prop: poho mon}) and the proof of Theorem \ref{full-rigidity-proof} can be carried out exactly as above. The fact that the volume of $ \Mb^n $ is infinite, which is used when Lemma \ref{tl-state} is invoked, is an immediate consequence of the convexity of $  \mathcal{V} $ (and it is in any case always true if $ \Mb^n $ is stochastically incomplete). The same holds for Theorem \ref{teo-cs-interv}, as the stochastic-incompleteness assumption yields $ \Mb^n \not \equiv \R^n $, so \eqref{eq-prod} is again satisfied and from there on the proof can be repeated identically.
\end{proof}

\smallskip

\textbf{Acknowledgments.} M.M.~was supported by the PRIN Project ``Direct and Inverse Problems for Partial Differential Equations: Theoretical Aspects and Applications'' (grant no.~201758MTR2, MIUR, Italy). The authors are partially supported by the INdAM-GNAMPA group.



\end{document}